\documentclass{amsart}
\usepackage{amsmath,amsthm,amsfonts,amssymb,mathtools,mathrsfs,url}
  \usepackage{paralist}
  \usepackage{graphics}
    \usepackage{epsfig}
\usepackage{graphicx}
\usepackage[colorlinks=true]{hyperref}
\usepackage{graphicx}
\usepackage{amsmath}
\usepackage{amssymb}
\usepackage[utf8]{inputenc}
\usepackage[french, main=english]{babel}
\usepackage{subcaption}
\usepackage{enumitem}
\usepackage{url}
\usepackage{xcolor}
\usepackage[toc,page]{appendix}
\usepackage{stmaryrd}
\usepackage{mathrsfs}
\usepackage{mathtools}
\usepackage[abbrev]{amsrefs}
\usepackage{tikz}
\usetikzlibrary{patterns.meta}
\usepackage{pgfplots}
\usepackage{lmodern}
\usepackage{cancel}

\pgfplotsset{every axis/.append style={
                    axis x line=middle,    
                    axis y line=middle,    
                    axis line style={->}, 
                    xlabel={$x$},          
                    ylabel={$v$},          
            }}

\numberwithin{equation}{section}

\hypersetup{urlcolor=blue, citecolor=green}

\mathtoolsset{showonlyrefs=true}

\calclayout

\newcommand{\R}{\mathbb R} 
\newcommand{\N}{\mathbb N}

\newcommand{\C}{\mathbb C}

\newcommand{\1}{\mathbf 1}
\newcommand{\D}{\mathrm{d}}
\newcommand{\om}{\omega}

\newcommand{\e}{\mathrm{e}}

\newcommand{\hop}{\vskip.3cm\noindent} 
\newcommand{\hip}{\vskip.1cm\noindent}

\newtheorem{thm}{Theorem}[section]   
\newtheorem{prop}[thm]{Proposition}
\newtheorem{lem}[thm]{Lemma}
\newtheorem{defi}[thm]{Definition}
\newtheorem{rema}[thm]{Remark}
\newtheorem{hypo}[thm]{Hypothesis}

\DeclareRobustCommand{\SkipTocEntry}[5]{}

\title[Spectrum of a semiclassical random walk with a general potential]{Spectral analysis of a semiclassical random walk associated to a general confining potential} 
\author[T. Normand]{Thomas Normand}
\email{thomas.normand@univ-nantes.fr}
\date{}

\begin{document}

\begin{abstract}
We consider a semiclassical random walk with respect to a probability measure associated to a potential with a finite number of critical points.
We recover the spectral results from \cite{BoHeMi} on the corresponding operator in a more general setting and with improved accuracy.
In particular we do not make any assumption on the distribution of the critical points of the potential, in the spirit of \cite{michel}.
Our approach consists in adapting the ideas from \cite{michel} to the recent \emph{gaussian quasimodes} framework which appears to be more robust than the usual methods, especially when dealing with non local operators.
\end{abstract}

\maketitle

\tableofcontents

\section{Introduction}

\subsection{Motivation}

Consider the probability measure $\D \mu_h(x)=Z_h \e^{-W(x)/h}\D x$ on $\R^d$, where $W:\R^d \to \R$ is a smooth function, $h>0$ is a small parameter and $Z_h$ is a normalization factor; as well as the Markov kernel
$$t_h(x,\D y)=\frac{1}{\mu_h(B(x,h))}\1_{|x-y|<h}\D \mu_h (y).$$
This kernel describes the following random walk: if at time $n\in \N$ the walk is in $x_n$, then the point $x_{n+1}$ is chosen in the small ball $B(x_n,h)$ uniformly at random with respect to d$\mu_h$.
Note that if $W$ is a Morse function, the density $\e^{-W/h}$ concentrates at scale $\sqrt h$ around the local minima of $W$, while the moves of the walk are at scale $h$.
We can also associate to $t_h(x,\D y)$ the bounded operator on $L^\infty(\R^d)$
$$\mathbf T_hf(x)=\int_{\R^d}f(y) t_h(x,\D y)=\frac{1}{\mu_h(B(x,h))}\int_{|x-y|<h}f(y)\D \mu_h(y).$$
It features the Markov property $\mathbf T_h(1)=1$ and leaves the subspace of continuous functions going to zero at infinity invariant.
Its adjoint is defined by duality on the set of bounded measures by
$$\mathbf T_h^* (\D \nu)=\Big(\int_{\R^d} \1_{|x-y|<h} \mu_h(B(y,h))^{-1} \D \nu(y) \Big) \D \mu_h.$$
If the starting point $x_0$ of the random walk is distributed according to $\D \nu$, then $x_1$ is distributed according to $\mathbf T_h^* (\D \nu)$.
More genrally, $x_n$ is distributed according to $(\mathbf T_h^* )^n(\D \nu)$.
One can easily check that $\mathbf T_h^*$ admits the invariant measure
$$\D \nu_{h,\infty}=\tilde Z_h \mu_h(B(x,h))\D \mu_h(x)$$
where $\tilde Z_h$ is a normalization factor.

This Markov process was studied in \cite{BoHeMi} where the authors showed the convergence of $(\mathbf T_h^* )^n(\D \nu)$ towards $\D \nu_{h,\infty}$ when $n\to \infty$ and gave some precise information on the speed of convergence.
More precisely, they showed by an interpolation argument that $\mathbf T_h^*$ extends as a bounded self-adjoint operator on $L^2(\D \nu_{h,\infty})$ with norm 1 
and gave an accurate description of its spectrum near $1$.
It enables the authors to state in \cite[Corollary 1.4]{BoHeMi} the existence of metastable states, in the spirit for instance of the works \cite{BoEcGaKl,BoGaKl}.
This indicates in particular that a very large number of iterations is required to make sure that the system returns to equilibrium.

The convergence of Markov chains to stationary distributions is a natural subject of interest.
Such information is for instance used to sample a given probability in order to implement Monte-Carlo methods (see \cite{LeRoSt}).
Some first results for discrete time processes on continuous state space were obtained in \cite{DiLe,DiLeMi,LeMi,GuMi}.
In these papers, the spectral gap of the studied operators is of order $h^2$ as the probability $\D \mu_h$ does not depend on $h$.
In our case, we have to deal with an exponentially small spectral gap.
 
The precise asymptotics of this gap and more generally of the eigenvalues exponentially close to 1 were obtained in \cite{BoHeMi} thanks to the exhibition of a supersymmetric-type structure for the operator, allowing to see it as a Hodge-Witten Laplacian on 0-forms for some pseudo-differential metric.
The idea is then to study both the associated derivative acting from 0-forms into 1-forms and its adjoint with the help of basic quasimodes.
This used to be a common method to study the small spectrum of semiclassical operators (see for instance \cite{HeKlNi,HHS}).
However, our goal here will be to give precise spectral asymptotics for the operator $\mathbf T_h^*$ through a more recent approach developed for the study of Fokker-Planck type differential operators in \cite{LPMichel,BonyLPMichel} and adapted to a non local framework in \cite{me}.
This approach consists in directly constructing a family of accurate quasimodes for our operator that we call \emph{gaussian quasimodes}.

The results found in the literature about the spectrum of semiclassical operators associated to some potential are often established for some particular potentials or at least satisfying a non degeneracy assumption (see for instance \cite[assumption (Gener)]{BonyLPMichel} or \cite[Hypothesis 3.11]{me}); except in \cite{michel} where the case of general potentials was treated for the Witten Laplacian.
In this spirit, the aim of this work is to adapt the ideas introduced in \cite{michel} to a non local framework and to the use of gaussian quasimodes to obtain a sharp description of the spectrum of $\mathbf T_h^*$ near 1 without the usual non degeneracy assumption on the potential $W$.

\subsection{Setting}

Before we can state the properties of the potential $W$ and the associated operator, let us introduce a few notations of semiclassical microlocal analysis which will be used in all this paper.
These are mainly extracted from \cite[chapter 4]{Zworski}.
We will denote $\xi\in \R^{d}$ 
the dual variable of $x$ 
and consider the space of semiclassical symbols 
\begin{align*}
S^0 \big(\langle x \rangle^k \langle \xi_m \rangle^{k'} \langle \xi_p \rangle^{k'}\big)=\big\{a^h \in \mathcal C^\infty(\R^{2d}) \, ; \, \forall \alpha \in \N^{2d}, \exists \, C_\alpha>0\;
		\emph{\text{such that }} |\partial^\alpha a^h(x,\xi)|\leq C_\alpha \langle x \rangle^k \langle \xi_m \rangle^{k'} \langle \xi_p \rangle^{k'}\big\}
\end{align*}
where $m$, $p\in \llbracket 1,d \rrbracket$ and $k$, $k'\in \R$. 
Given a symbol $a^h \in S^0(\langle x \rangle^k \langle \xi_m \rangle^{k'} \langle \xi_p \rangle^{k'})$, we define the associated semiclassical pseudo-differential operator for the Weyl quantization acting on functions $u \in \mathcal S(\R^{d})$ by
$$\mathrm{Op}_h(a^h)u(x)=(2\pi h)^{-d}\int_{\R^{d}} \int_{\R^{d}}\e^{\frac ih (x-x')\cdot \xi}a^h\Big(\frac{x+x'}{2},\xi\Big)u(x')\,\D x'\D \xi$$
where the integrals may have to be interpreted as oscillating integrals.
\hip
Our only hypothesis on the potential $W$ is the following.

\begin{hypo}\label{W}
The potential $W$ is a smooth Morse function with values in $\R$ such that
\begin{align}\label{confin}
\e^{-W/h}\in L^2(\R^{d}),  \quad \quad \lim_{|x|\to +\infty}W(x)=+\infty \quad \text{and} \quad  W\in S^0(\langle x\rangle^\eta) 
\end{align}
for some $\eta \in \N$.
Moreover, for every $0\leq k\leq d$, the set of critical points of index $k$ of $W$ that we denote $\mathcal U^{(k)}$ is finite and we set 
\begin{align}\label{n0}
n_0=\# \mathcal U^{(0)}.
\end{align}
Finally, we will suppose that $n_0\geq 2$.
\end{hypo}
\hip
For the operator $\mathbf T_h^*$ acting on $L^2(\D \nu_{h,\infty})$, as it is more convenient to work with the standard Lebesgue measure, and since we want to study its spectrum near 1, we choose to consider the operator
\begin{align}\label{ph}
P_h=\mathrm{Id}-\mathcal M_h^{1/2} \circ \mathbf T_h^*\circ \mathcal M_h^{-1/2}
\end{align}
instead, where $\mathcal M_h^{1/2}$ stands for both the square root of the function
$$\mathcal M_h(x)=\tilde Z_h Z_h \mu_h(B(x,h))\e^{-W(x)/h}$$
and the associated unitary operator from $L^2(\D \nu_{h,\infty})$ to $ L^2(\R^d)$.
Our goal is now to give a sharp description of the small spectrum of $P_h$ acting on $L^2(\R^d)$.
We will actually be able to treat the case of some slightly more general operators $P_h$ than the one given by \eqref{ph}.
In order to focus the difficulties mainly on the topology of the potential $W$, we will consider some operators $P_h$ which still present some nice properties, even though it makes no doubt that we could adopt an even more general setting.  

More precisely, let us introduce some notions of expansions of symbols: we will say that 
\begin{align}\label{exph}
a^h \sim \sum_{j \geq 0} h^j a_j \quad\text{ in }S^0(1)
\end{align}
if $(a_j)_{j \geq 0}\subset S^0(1)$ is a family of symbols independent of $h$ and such that for all $N\in \N$, 
$$ a^h-\sum_{j=0}^{N-1}h^j a_j = O_{S^0(1)}(h^N).$$
We also need to introduce the notion of analytic symbols.
\begin{defi}\label{symbolo}
For $\kappa>0$, let us introduce the set 
$$\Sigma_\kappa=\{z\in \C \, ; \, |\mathrm{Im}\,z|< \kappa\}^d\subset \C^d.$$
We denote $S^0_\kappa(\langle x \rangle^k \langle \xi_m \rangle^{k'} \langle \xi_p \rangle^{k'})$ the space of symbols $a^h\in S^0(\langle x \rangle^k \langle \xi_m \rangle^{k'} \langle \xi_p \rangle^{k'})$ such that:
\begin{enumerate}[label=$\mathrm{(\roman*)}$]
\item For all 
$x\in \R^{d}$, 
$a^h(x,\cdot)$ is analytic on $\Sigma_\kappa$ \label{holom}
\item For all $\beta \in \N^{d}$, there exists $C_\beta>0$ such that $|\partial_{x}^{\beta} a^h |\leq C_\beta \langle x \rangle^k \langle \xi_m \rangle^{k'} \langle \xi_p \rangle^{k'}$
on $\R^{d}\times \Sigma_\kappa.$ \label{majobande}
\end{enumerate}
We will also use the notation $a^h=O_{S^0_\kappa(\langle x \rangle^k \langle \xi_m \rangle^{k'} \langle \xi_p \rangle^{k'})}(h^N)$ to say that for all $\alpha \in \N^{2d}$, there exists $C_{\alpha, N}$ such that
$| \partial^\alpha  a^h |\leq C_{\alpha, N} \, h^N \langle x \rangle^k \langle \xi_m \rangle^{k'} \langle \xi_p \rangle^{k'}$
on $\R^{d}\times \Sigma_\kappa$.
\end{defi}
\hip
Using the Cauchy-Riemann equations, we see that item \ref{holom} from Definition \ref{symbolo} implies that for all $\beta \in \N^{d}$ and $x\in \R^{d}$, the functions $\partial^\beta_{x} a^h(x,\cdot)$ are also analytic on $\Sigma_\kappa$.
Besides, the Cauchy formula implies that for any $\tilde \kappa<\kappa$, $\alpha \in \N^d$ and $\beta\in \N^{d}$, there exists $C_{\alpha,\beta}$ such that
$$|\partial_\xi^\alpha\partial_{x}^\beta a^h |\leq C_{\alpha,\beta} \langle x \rangle^k \langle \xi_m \rangle^{k'} \langle \xi_p \rangle^{k'} \qquad \text{ on } \R^{d}\times \Sigma_{\tilde \kappa}$$
i.e up to taking $\kappa$ smaller, item \ref{majobande} from Definition \ref{symbolo} can be extended to $\beta\in \N^{2d}$.
When dealing with analytic symbols, our notion of expansion becomes 
\begin{align}\label{exph}
a^h \sim \sum_{j \geq 0} h^j a_j \quad\text{ in }S^0_\kappa(\langle x \rangle^k \langle \xi_m \rangle^{k'} \langle \xi_p \rangle^{k'})
\end{align}
if $(a_j)_{j \geq 0}\subset S^0_\kappa(\langle x \rangle^k \langle \xi_m \rangle^{k'} \langle \xi_p \rangle^{k'})$ is a family of symbols independent of $h$ and such that for all $N\in \N$, 
$$ a^h-\sum_{j=0}^{N-1}h^j a_j = O_{S^0_\kappa(\langle x \rangle^k \langle \xi_m \rangle^{k'} \langle \xi_p \rangle^{k'})}(h^N).$$
We now extend these notions to matrix valued symbols: if $n_1$, $n_2\in \llbracket 1,d \rrbracket$ and
$$q^h=(q_{m,p})\mathop{}_{\substack{1\leq m \leq n_1 \\  1\leq p \leq n_2}}$$
\sloppy is a matrix of functions such that each $q_{m,p}\in S^0_\kappa(\langle x \rangle^k \langle \xi_m \rangle^{k'} \langle \xi_p \rangle^{k'})$, we say that $q^h\in \mathcal M_{n_1,n_2}\big(S^0_\kappa(\langle x \rangle^k \langle \xi_m \rangle^{k'} \langle \xi_p \rangle^{k'})\big)$ and we denote 
$$\mathrm{Op}_h(q^h)=\Big(\mathrm{Op}_h(q_{m,p})\Big)\mathop{}_{\substack{1\leq m \leq n_1 \\  1\leq p \leq n_2}}.$$
Even though it does not appear in the notations, the function $q_{m,p}$ may also depend on $h$.
The notation 
$$ q^h=O_{\mathcal M_{n_1,n_2}\big(S^0_\kappa(\langle x \rangle^k \langle \xi_m \rangle^{k'} \langle \xi_p \rangle^{k'})\big)}(h^N)$$
means that for all $(m,p)\in \llbracket 1,n_1 \rrbracket\times \llbracket 1,n_2 \rrbracket$, the symbol $q_{m,p}$ is $O_{S^0_\kappa(\langle x \rangle^k \langle \xi_m \rangle^{k'} \langle \xi_p \rangle^{k'})}(h^N)$.
Furthermore, the notion of expansion $q\sim \sum_{n \geq 0} h^n q_n$ in $\mathcal M_{n_1,n_2}\big(S^0_\kappa(\langle x \rangle^k \langle \xi_m \rangle^{k'} \langle \xi_p \rangle^{k'})\big)$ is a straightforward adaptation of the one for scalar symbols.
\hop
These notions enable us to introduce the class of operators that we will consider.
Let us denote $d_W$ the twisted derivative 
$$d_W=h\nabla +\nabla W.$$
We also use the standard notation $\mathcal M_d(\R)$ for the set of all $d$-by-$d$ real matrices.
\begin{hypo}\label{hyporw}
We assume that $P_h$ is a bounded operator such that
$$P_h=a^h \circ \tilde P_h \circ a^h$$
with 
\begin{align}\label{factolab}
\tilde P_h=d_W^* \circ \widehat Q \circ d_W
\end{align}
where $\widehat Q=\mathrm{Op}_h(q^h)$ is a self-adjoint, non negative pseudo-differential operator and $a^h\sim \sum_{j\geq 0}h^j a_j$ in $S^0(1)$ is a positive symbol such that $(a^h)^{-1}\in S^0(1)$ and $a_0(x)=1$ as soon as $x$ is a critical point of $W$.
Moreover, 
\begin{enumerate}[label=\alph*)]
\item $P_h$ admits $0$ as a simple eigenvalue.
\item There exist $c>0$ and $h_0>0$ such that for all $0<h\leq h_0$, we have that 
	$\mathrm{Spec}(P_h)\cap [0, ch]$ consists of exactly $n_0$ eigenvalues (counted with algebraic multiplicity) 
	that are exponentialy small with respect to $1/h$.
\item For all $x,$ $\xi\in \R^d$, we have $q^h(x,-\xi)=q^h(x,\xi)$.
\item The symbol $q^h$ is analytic in the variable $\xi$. More precisely, there exists $\kappa>0$ such that
$$q^h\sim \sum_{n\geq 0}h^n q_n \text{ in the space of analytic symbols }\mathcal M_d\big(S^0_\kappa (\langle \xi_m \rangle^{-1} \langle \xi_p \rangle^{-1} )\big).$$
\item
There exists a constant $\varrho>0$ such that for all saddle point $\mathbf s$ of $W$, we have $q_0(\mathbf s,0)=\varrho \,\mathrm{Id}$. \label{q0id}
\end{enumerate}
In particular, the resolvent estimate
\begin{align}\label{resest}
(z-P_h)^{-1}=O(h^{-1})
\end{align}
is satisfied on $|z|=ch$.
\end{hypo}
\hip
Throughout the paper, we will work under Hypothesis \ref{hyporw}.

It is shown in \cite{BoHeMi} (Theorem 1.1 as well as Corollary 3.1) that the random walk operator \eqref{ph} studied in this reference is an example of a non local operator satisfying Hypothesis \ref{hyporw} with $\varrho=(2d+4)^{-1}$.
We could for instance easily treat the case where the identity matrix is replaced by a positive definite matrix in item \ref{q0id} (as it is done in \cite{me}), but our real interest here are the considerations brought by the potential $W$.
\hop

The main result of this paper is Theorem \ref{thmgene} in which we give a sharp description of the small eigenvalues of $P_h$. 
As we have not yet introduced all the technical objects involved in this statement, let us for the moment give a rather vague version of this result.

\begin{thm}
Under Hypothesis \ref{hyporw}, there exist $p\in \N$ and a finite set $\mathcal A$ both explicit as well as some positive definite matrices $(\mathcal M^{\alpha,j}_h)_{1\leq j\leq p\,;\, \alpha\in \mathcal A}$ depending on $\varrho$ from Hypothesis \ref{hyporw} and admitting a classical expansion whose first term is given in Theorem \ref{thmgene} such that
$$\Big(\mathrm{Spec}(P_h)\cap [0, ch]\Big)\subset h  \bigcup_{\alpha\in \mathcal A} \bigcup_{j=1}^p \e^{-2\hat S_j/h} \Big( \mathrm{Spec} \big( \mathcal M^{\alpha,j}_h + D\big(0,O(h^\infty)\big)\Big) $$
where $\{\hat S_1<\dots < \hat S_p\}$ are the finite values taken by the map $S$ introduced in Definition \ref{j et s}.
\hop
In particular, when $P_h$ is given by \eqref{ph}, we have 
$$\Big(\mathrm{Spec}(\mathbf T_h^*)\cap [1- ch, 1]\Big)\subset \left( 1- h  \bigcup_{\alpha\in \mathcal A} \bigcup_{j=1}^p \e^{-2\hat S_j/h} \Big( \mathrm{Spec} \big( \mathcal M^{\alpha,j}_h + D\big(0,O(h^\infty)\big)\Big)\right)$$
with $\varrho=(2d+4)^{-1}$.
\end{thm}

It is an improvement of the result of \cite{BoHeMi} in two ways: first, we treat in the spirit of \cite{michel} the case of general potentials satisfying only Hypothesis \ref{W} instead of Hypotheses 1 and 2 from \cite{BoHeMi}.
Moreover, we establish complete asymptotic expansions of the small eigenvalues, i.e the remainder terms in the prefactors are of order $h^\infty$ and not of order $h$ as in \cite{BoHeMi}.
This work can also be seen as an adaptation of the considerations from \cite{michel} to a non local framework and to the gaussian quasimodes approach, as there exist some operators for which it is the only known approach to succeed (see for instance the Boltzmann operators from \cite{me,meRL}).

\section{General labeling of the potential minima}

Before we can construct our quasimodes, we need to recall the general labeling of the minima which originates from \cite{HeKlNi} and was generalized in \cite{HHS11}, as well as the topological constructions that go with it.
Here we only introduce the essential objects and omit the proofs.
For more details, we refer to \cite{michel,me}. 
Recall that we denote 
\begin{align}\label{UkY}
\mathcal U^{(k)} \text{ the critical points of }W \text{ of index }k.
\end{align}
For shortness, we will write \og CC \fg{} instead of \og connected component \fg{}.
The constructions rely on the following fundamental observation which is an easy consequence of the Morse Lemma (see for instance \cite{me}, Lemma 3.1 for a proof):
\begin{lem}\label{1.4}
If $x\in \mathcal U^{(1)}$, then there exists $r_0>0$ such that for all $0<r<r_0$, $x$ has a connected neighborhood $\mathcal O_r$ in $B(x,r)$ such that $\mathcal O_r\cap \{W< W(x)\}$ has exactly 2 CCs.\\
\end{lem}
\vskip -0.4cm
\hip
It motivates the following definition:
\begin{defi}\label{ssv}
\begin{enumerate}
\item We say that $x\in \mathcal U^{(1)}$ is a separating saddle point and we denote $x\in \mathcal V^{(1)}$ if for every $r>0$ small enough, the two CCs of $\mathcal O_r\cap \{W< W(x)\}$ are contained in different CCs of $\{W< W(x)\}$.
\item We say that $\sigma \in \R$ is a separating saddle value if $\sigma \in W(\mathcal V^{(1)})$.
\end{enumerate}
\end{defi}

It is known (see for instance \cite{me}, Lemma 3.4) that $\mathcal V^{(1)}\neq \emptyset$. Let us then denote $\sigma_2>\dots>\sigma_N$ where $N\geq 2$ the different separating saddle values of $W$ and for convenience we set $\sigma_1=+\infty$.
For $\sigma \in \R\cup\{+\infty\}$, let us denote $\mathcal C_\sigma$ the set of all the CCs of $\{W<\sigma\}$.
We call \emph{labeling} of the minima of $W$ any injection $\mathcal U^{(0)}\to \llbracket 1,N \rrbracket \times \N^*$ which we denote for shortness $(\mathbf m_ {k,j})_{k,j}$.
Given a labeling $(\mathbf m_{k,j})_{k,j}$ of the minima of $W$, we denote for $k\in \llbracket 1,N \rrbracket$ 
$$\mathtt U^{(0)}_k=\big\{\mathbf m_{k',j}\,;\, 1\leq k'\leq k \, , \, j\in \N^*\big\}\cap \big\{W<\sigma_k\big\}$$
and we say that the labeling is \emph{adapted} to the separating saddle values if for all $k\in \llbracket 1,N \rrbracket$, each $\mathbf m_{k,j}$ is a global minimum of $W$ restricted to some CC of $\{W<\sigma_k\}$ and the map 
\begin{align}\label{TkY}
T_k: \mathtt U^{(0)}_k \to \mathcal C_ {\sigma_k}
\end{align}
sending $\mathbf m \in \mathtt U^{(0)}_k$ on the element of $\mathcal C_ {\sigma_k}$ to which it belongs is bijective.
In particular, it implies that each $\mathbf m_{k,j}$ belongs to $\mathtt U^{(0)}_k$.
Such labelings exist, one can for instance easily check that the usual labeling procedure presented in \cite{HHS11} is adapted to the separating saddle values.
From now on, we fix a labeling $(\mathbf m_{k,j})_ {k,j}$ adapted to the separating saddle values of $W$.

\begin{defi}\label{j et s}
Recall the notation \eqref{UkY} and Definition \ref{ssv}.
We define the following mappings:
\begin{enumerate}[label=\textbullet]
\item $E: \mathcal U^{(0)}\xrightarrow{{}\quad{}} \mathcal P(\R^{d})$\\
	${} \; \mathbf m_{k,j} \xmapsto{{}\quad{}} T_k(\mathbf m_{k,j})$\\
where $T_k$ is the map defined in \eqref{TkY}.
\item 
$\mathbf j:\mathcal U^{(0)}\to \mathcal P\big(\mathcal V^{(1)}\cup \{\mathbf s_1\}\big)$\\
given by $\mathbf j(\mathbf m_{1,1})=\{\mathbf s_1\}$ where $\mathbf s_1$ is a fictive saddle point such that $W(\mathbf s_1)=\sigma_1=+\infty$; and for $2\leq k\leq N$, $\mathbf j(\mathbf m_{k,j})=\partial E(\mathbf m_{k,j})\cap \mathcal V^{(1)}$ which is not empty (see for instance Lemma 3.5 from \cite{me}) and included in $\{W=\sigma_k\}$.
\item
$\boldsymbol \sigma:\mathcal U^{(0)}\to W(\mathcal V^{(1)})\cup \{\sigma_1\}$\\
${} \quad \mathbf m \mapsto W(\mathbf j(\mathbf m))$\\
where we allow ourselves to identify the set $W(\mathbf j(\mathbf m))$ and its unique element in $W(\mathcal V^{(1)})\cup \{\sigma_1\}$.
\item
$S: \mathcal U^{(0)}\xrightarrow{{}\quad{}} ]0,+\infty]$\\
	${} \quad \mathbf m \xmapsto{{}\quad{}} \boldsymbol \sigma(\mathbf m)-W(\mathbf m)$.
\end{enumerate}
\end{defi}
\hip
We now introduce some material from \cite{michel}.
Let us denote
$$\underline{\mathbf m}=\mathbf m_{1,1} \qquad \text{and} \qquad \underline{\mathcal U}^{(0)}=\mathcal U^{(0)}\backslash \{\underline{\mathbf m}\}$$
and define for $\mathbf m=\mathbf m_{k,j}\in \underline{\mathcal U}^{(0)}$
$$\widehat{\mathbf m}=T^{-1}_{k-1}\big(E_-(\mathbf m)\big)$$
where $E_-(\mathbf m)$ is the element of $\mathcal C_ {\sigma_{k-1}}$ containing $\mathbf m$.
Since $\widehat{\mathbf m}$ and $\mathbf m$ both belong to $E_-(\mathbf m)$, we have $W(\widehat{\mathbf m})\leq W(\mathbf m)$ and $\widehat{\mathbf m}\in \mathtt U^{(0)}_k$.
\begin{defi}
\begin{enumerate}[label=\textbullet]
\item A minimum $\mathbf m\in \underline{\mathcal U}^{(0)}$ is said to be of type I if $W(\widehat{\mathbf m})< W(\mathbf m)$.
Otherwise (i.e when $W(\widehat{\mathbf m})= W(\mathbf m)$), we say that $\mathbf m$ is of type II.
\item We define an equivalence relation $\mathcal R$ on $\underline{\mathcal U}^{(0)}$ by $\mathbf m \mathcal R \mathbf m'$ if and only if the two following conditions are satisfied:
	\begin{enumerate}[label=\roman*)]
	\item $\boldsymbol \sigma (\mathbf m)=\boldsymbol \sigma (\mathbf m')=\sigma_k$
	\item There exist some minima $\mathbf m^1, \dots , \mathbf m^K$ such that $\mathbf m^1=\mathbf m$, $\mathbf m^K=\mathbf m'$ and for all $1\leq n\leq K-1$, we have $\overline{T_k(\mathbf m^n)}\cap \overline{T_k(\mathbf m^{n+1})}\neq \emptyset$ with 
		$$\mathbf m^n\in \{\mathbf m_{k,j} \, ; \, j\in \N^*\}\cup\{\widehat{\mathbf m}_{k,j} \, ; \, j\in \N^* \text{ \emph{and} } \mathbf m_{k,j} \text{ \emph{is of type II}}\}.$$  
	\end{enumerate}
\end{enumerate}
\end{defi}
\hip
We denote the associated equivalence classes $(\mathcal U^{(0)}_\alpha)_{\alpha\in \mathcal A}$, where $\mathcal A$ is a finite set. 
It is shown in \cite{michel} (Proposition 2.6) that for $\mathbf m\mathcal R \mathbf m'$, we have $\boldsymbol\sigma (\mathbf m)=\boldsymbol\sigma (\mathbf m')$ and $\widehat{\mathbf m}=\widehat{\mathbf m}'$.
Finally, for $\mathbf m \in \mathcal U^{(0)}_\alpha$, we put
$$\boldsymbol \sigma(\alpha)=\boldsymbol \sigma(\mathbf m) \qquad \text{and} \qquad   \widehat{\mathcal U}^{(0)}_\alpha=\mathcal U^{(0)}_\alpha \cup \{\widehat{\mathbf m}\}$$ 
which do not depend on the choice of $\mathbf m \in \mathcal U^{(0)}_\alpha$, as well as
$$E^\alpha(\mathbf m)=T_{k(\alpha)}(\mathbf m) \qquad \text{and} \qquad   \mathbf j^\alpha(\mathbf m)=\partial E^\alpha(\mathbf m)\cap \mathcal V^{(1)}$$
for $\mathbf m \in \widehat{\mathcal U}^{(0)}_\alpha$, where $k(\alpha)$ is such that $\boldsymbol \sigma(\alpha)=\sigma_{k(\alpha)}$.
When $\mathbf m \in \mathcal U^{(0)}_\alpha$, we have $E^\alpha(\mathbf m)=E(\mathbf m)$ but for $\widehat{\mathbf m} \in \widehat{\mathcal U}^{(0)}_\alpha\cap \mathcal U^{(0)}_{\alpha'}$, we have $E(\widehat{\mathbf m})=E^{\alpha'}(\widehat{\mathbf m})\neq E^{\alpha}(\widehat{\mathbf m})$.

\section{Gaussian quasimodes}\label{sectquasim}

Throughout the paper, for $d'\in \N^*$, $\Omega \subseteq \R^{d'}$ and $a \in \mathcal C^{\infty}(\Omega)$ a function depending on $h$ and such that for all $\beta\in \N^{d'}$ we have $\partial^\beta a =O_{L^\infty}(1)$, we will say that $a \in \mathcal C^{\infty}(\Omega)$ admits a classical expansion on $\Omega$ and denote $a \sim \sum_{j \geq 0}h^j a_j$, where  $(a_j)_{j \geq 0}\subset \mathcal C^{\infty}(\Omega)$ are independent of $h$, provided that for all $\beta \in \N^{d'}$ and $N \in \N$, there exists $C_{\beta, N}$ such that 
$$\Big\|\partial^\beta \Big(a-\sum_{j = 0}^{N-1}h^j a_j\Big)\Big\|_{\infty, \Omega}\leq C_{\beta, N} h^N.$$
It implies in particular that $\partial^\beta a_j=O_{L^\infty}(1)$.
From now on, the letter $r$ will denote a small universal positive constant whose value may decrease as we progress in this paper (one can think of $r$ as $1/C$).
\hop

We are now going to introduce some quasimodes for the operator $P_h$.
With the notations from Hypothesis \ref{hyporw}, our approach consists in constructing some gaussian quasimodes for the factorised operator $\tilde P_h$ in the spirit of \cite{BonyLPMichel,me} and multiply those by $(a^h)^{-1}$.

Let $\alpha\in \mathcal A$ and $\mathbf m  \in \widehat{\mathcal U}^{(0)}_\alpha$. For each $\mathbf s \in \mathbf j^\alpha(\mathbf m)$ we introduce a function 
$\ell^{\mathbf s,\mathbf m}$ that will appear in our quasimodes.
Note that thanks to Lemma \ref{1.4}, there are at most two functions $\ell^{\mathbf s,\mathbf m}$ and $\ell^{\mathbf s,\mathbf m'}$ associated to a saddle point $\mathbf s \in \mathcal V^{(1)}$.
Our goal will be to find some functions $\ell^{\mathbf s,\mathbf m}$ such that our quasimodes are the most accurate possible.
In order to begin the computations that will yield the equations that the function $\ell^{\mathbf s,\mathbf m}$ should satisfy, we will for the moment assume that it satisfies the following:
\hop

\begin{minipage}{0.02\linewidth}
\begin{align}\label{hypol}
\,
\end{align}
\end{minipage}
\begin{minipage}{0.93\linewidth}
\begin{enumerate}[label=\alph*)]
\item $\ell^{\mathbf s,\mathbf m}$ is a smooth real valued function on $\R^{d}$ whose support is contained in $B(\mathbf s,3r)$ 
\label{hypol1}
\item $\ell^{\mathbf s,\mathbf m}$ admits a classical expansion $\ell^{\mathbf s,\mathbf m}\sim \sum h^j\ell^{\mathbf s,\mathbf m}_j$ on $B(\mathbf s,2r)$ 
\item $\ell^{\mathbf s,\mathbf m}_0$ vanishes at $\mathbf s$ \label{hypoln-2}
\item $\mathbf s$ is a local minimum of the function $W+(\ell_0^{\mathbf s,\mathbf m})^2/2$ which is non degenerate \label{hypoln-1}
\item the functions $\theta^\alpha_{\mathbf m,h}$ (which depends on $\ell^{\mathbf s,\mathbf m}$) and $\chi_{\alpha}$ that we will introduce in \eqref{thetainte}-\eqref{chim} are such that $\theta^\alpha_{\mathbf m,h}$ is smooth on a neighborhood of supp $\chi_{\alpha}$. \label{hypoln}
\end{enumerate}
\end{minipage}
\hop
Once we will have found the desired function $\ell^{\mathbf s,\mathbf m}$, we will see in Proposition \ref{lexist} that these assumptions are actually satisfied.
Denote $\zeta \in \mathcal C^{\infty}_c(\R, [0,1])$ an even cut-off function supported in $[-\gamma,\gamma]$ that is equal to $1$ on $[-\gamma/2,\gamma/2]$ where $\gamma>0$ is a parameter to be fixed later and 
\begin{align}\label{approxa}
A_h=\frac12 \int_\R \zeta(s) \e^{-\frac{s^2}{2h}}\D s=\int_0^{\gamma} \zeta(s) \e^{-\frac{s^2}{2h}}\D s=\frac{\sqrt{\pi h}}{\sqrt 2}(1+O(\e^{-c/h}))\qquad \text{for some }c>0.
\end{align}
We now define for each $\mathbf m  \in \widehat{\mathcal U}_\alpha^{(0)}$ a function $\theta_{\mathbf m,h}^\alpha$ as follows: if $x \in B(\mathbf s,r)\cap \{|\ell_0^{\mathbf s,\mathbf m}|\leq 2\gamma\}$ for some $\mathbf s \in \mathbf j^\alpha(\mathbf m)$, 
\begin{align}\label{thetainte}
\theta^\alpha_{\mathbf m,h}(x)=\frac12 \Big(1+A_h^{-1}\int_0^{\ell^{\mathbf s,\mathbf m}(x)}\zeta(s)\e^{-s^2/2h}\D s\Big)
\end{align} 
whereas we set 
\begin{align}\label{theta1}
\theta^\alpha_{\mathbf m,h}=1 \quad  \text{on } \Big(E^\alpha(\mathbf m)+B(0, \varepsilon) \Big)\backslash \Big(\bigsqcup_{\mathbf s\in\mathbf j^\alpha(\mathbf m)} \big(B(\mathbf s,r)\cap \{|\ell_0^{\mathbf s,\mathbf m}|\leq 2\gamma\}\big) \Big)
\end{align}
with $\varepsilon(r)>0$ to be fixed later and 
\begin{align}\label{theta0}
\theta_{\mathbf m,h}=0 \quad  \text{everywhere else.}
\end{align}
Note that $\theta_{\mathbf m,h}^\alpha$ takes values in $[0,1]$ and that we have 
\begin{align}\label{supptheta}
\mathrm{supp}\, \theta^\alpha_{\mathbf m,h} \subseteq E^\alpha(\mathbf m)+B(0,\varepsilon') 
\end{align}
where $\varepsilon'=\max(\varepsilon,r)$.\\
Denote now $\Omega_\alpha$ the CC of $\{W\leq \boldsymbol \sigma(\alpha)\}$ containing $\widehat{\mathcal U}_\alpha^{(0)}$.
The CCs of $\{W\leq \boldsymbol \sigma(\alpha)\}$ are separated so for $\varepsilon>0$ small enough, there exists $\tilde \varepsilon>0$ such that 
$$\min\,\big\{W(x)\, ; \, \D \big (x,\Omega_\alpha\big)=\varepsilon\big\}=\boldsymbol \sigma(\alpha)+2 \tilde \varepsilon.$$
Thus the distance between $\{W \leq \boldsymbol \sigma(\alpha)+\tilde \varepsilon\}\cap \big(\Omega_\alpha+B(0,\varepsilon)\big)$ and $\partial \big(\Omega_\alpha+B(0,\varepsilon)\big)$ is positive and we can consider a cut-off function 
\begin{align}\label{chim}
\chi_\alpha\in \mathcal C^{\infty}_c(\R^{d},[0,1])
\end{align}
such that 
$$\chi_\alpha =1 \text{ on } \{W \leq \boldsymbol \sigma(\alpha)+\tilde \varepsilon\}\cap \big(\Omega_\alpha+B(0,\varepsilon)\big)$$
 and 
$$\mathrm{supp} \,\chi_\alpha \subset  \big(\Omega_\alpha+B(0,\varepsilon)\big).$$
To sum up, we have the following picture:

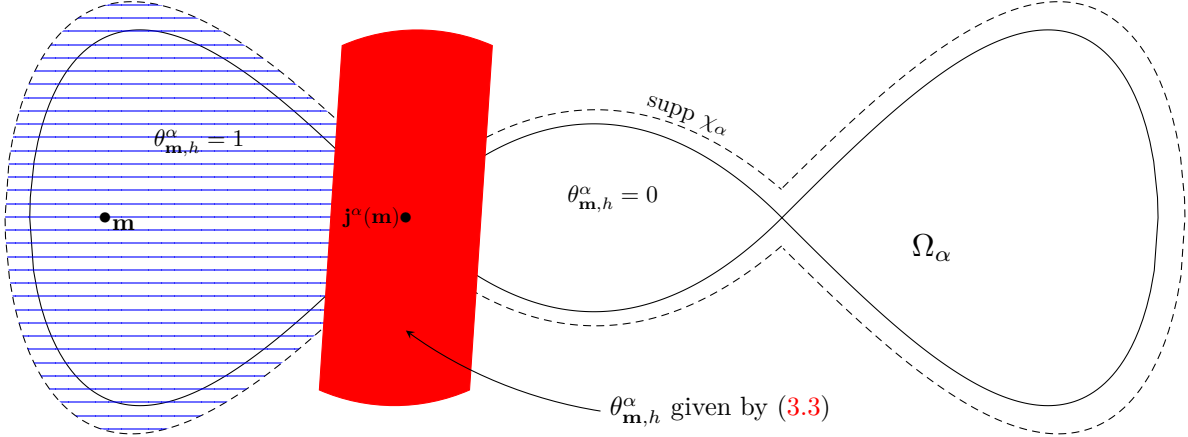
\begin{center}
\begin{tikzpicture}[scale=5]

\draw plot[samples=100,domain=1:2] (\x, {(\x-1)*sqrt(1-((\x-1))^2)}) 
		-- plot[samples=100,domain=2:1] (\x, {-(\x-1)*sqrt(1-((\x-1))^2)}) -- cycle ;
\draw plot[samples=100,domain=0:1] (\x, {\x*(1-\x)}) -- plot[samples=100,domain=1:0] (\x, {-\x*(1-\x)}) -- cycle ;
\draw plot[samples=100,domain=-1:0] (\x, {\x*sqrt(1-(\x)^2)}) -- plot[samples=100,domain=0:-1] (\x, {-\x*sqrt(1-(\x)^2)}) -- cycle ;

\begin{scope} [xshift=-6.5,scale=1.15]
\draw[densely dashed] plot[samples=100,domain=1.07:2] (\x, {(\x-1)*sqrt(1-((\x-1))^2)}) 
		-- plot[samples=100,domain=2:1.07] (\x, {-(\x-1)*sqrt(1-((\x-1))^2)}) ;
\end{scope}
\begin{scope} [xshift=-2,scale=1.15]
\draw[densely dashed] plot[samples=100,domain=0.07:0.93] (\x, {\x*(1-\x)});
\draw[densely dashed] plot[samples=100,domain=0.93:0.07] (\x, {-\x*(1-\x)});
\end{scope}
\begin{scope} [xshift=2.4,scale=1.15]
\draw[densely dashed] plot[samples=100,domain=-0.07:-1] (\x, {\x*sqrt(1-(\x)^2)}) -- plot[samples=100,domain=-1:-0.07] (\x, {-\x*sqrt(1-(\x)^2)});
\fill[pattern={Lines[angle=45,distance=5pt]},pattern color=blue,opacity=0.7] (-0.02,0) -- plot[samples=100,domain=-0.07:-1] (\x, {\x*sqrt(1-(\x)^2)}) -- plot[samples=100,domain=-1:-0.07] (\x, {-\x*sqrt(1-(\x)^2)}) -- cycle;
\end{scope}


\coordinate (ne) at (0.5*0.4+0.03,{0.5*sqrt(1-(0.4)^2)});
\coordinate (nw) at (-0.5*0.4,{0.5*sqrt(1-(0.4)^2)});
\coordinate (se) at (0.5*0.4,{-0.5*sqrt(1-(0.4)^2)});
\coordinate (sw) at (-0.5*0.4-0.03,{-0.5*sqrt(1-(0.4)^2)});
\filldraw[red,fill opacity=0.5] (ne) arc (1.16 r:{(pi-1.16) r}:0.5) -- (sw) arc ({-(pi-1.16) r}:-1.16 r:0.5) -- cycle;

\draw (0,0) node {$\bullet$};
\draw (-0.09,0) node[scale=0.8] {$\mathbf j^\alpha(\mathbf m)$};
\draw (-0.8,0) node {$\bullet$};
\draw (-0.75,-0.02) node[scale=0.9] {$\mathbf m$};
\draw (-0.55,0.2) node[scale=0.9] {$\theta^\alpha_{\mathbf m,h}=1$};
\draw (0.55,0.05) node[scale=0.9] {$\theta^\alpha_{\mathbf m,h}=0$};
\draw (1.4,-0.08) node[scale=1.2] {$\Omega_\alpha$};
\draw (0.75,0.27) node[scale=0.9,rotate=-19] {supp $\chi_\alpha$};

\draw[>=stealth,<-] (0,-0.3) arc ({(-0.7*pi) r}:{(-0.55*pi) r}:1.2) node[right] {$\theta^\alpha_{\mathbf m,h}$ given by \eqref{thetainte}};

\end{tikzpicture}
\end{center}
We also denote 
$$W_{\mathbf m}(x)=W(x)-W(\mathbf m)$$ 
and it is clear that 
\begin{align}\label{suppchi}
W_{\mathbf m}\geq S(\mathbf m)+\tilde \varepsilon \qquad \text{on the support of } \nabla \chi_\alpha \text{ as soon as }\mathbf m\in \mathcal U^{(0)}_\alpha.
\end{align}
Recalling the function $a^h$ from Hypothesis \ref{hyporw}, the global quasimode associated to the eigenvalue $0$ is
\begin{align}\label{noyau}
f_{\underline{\mathbf m},h}(x)=(a^h)^{-1}h^{-d/4}c_h(\underline{\mathbf m})\e^{-W_{\underline{\mathbf m}}(x)/h}\in \mathrm{Ker}\, P_h
\end{align}
which is an exact one, while for $\mathbf m \in \mathcal U_\alpha^{(0)}$, our quasimodes will be linear combinations of the functions
$$(a^h)^{-1}h^{-d/4} c^\alpha_h(\tilde{\mathbf m}) \chi_\alpha(x) \theta^\alpha_{\tilde{\mathbf m},h}(x) \e^{-W_{\tilde{\mathbf m}}(x)/h}$$
where $\tilde{\mathbf m}\in \widehat{\mathcal U}_\alpha^{(0)}$.
Here $c^\alpha_h(\mathbf m)$ and $c_h(\underline{\mathbf m})$ are normalization factors assuring that
\begin{align}\label{renorm}
\text{ }\quad \;\|(a^h)^{-1}h^{-d/4} c^\alpha_h(\tilde{\mathbf m}) \chi_\alpha(x) \theta^\alpha_{\tilde{\mathbf m},h}(x) \e^{-W_{\tilde{\mathbf m}}(x)/h}\|=1 \quad \; \text{and}\quad \; \|(a^h)^{-1}h^{-d/4}c_h(\underline{\mathbf m})\e^{-W_{\underline{\mathbf m}}(x)/h}\|=1.
\end{align}
In particular, thanks to Hypothesis \ref{hyporw} we have that for all $\mathbf m\in \widehat{\mathcal U}_\alpha^{(0)}$, the constant $c^\alpha_h(\mathbf m)$ (resp. the constant $c_h(\underline{\mathbf m})$) admits an asymptotic expansion whose first term is
\begin{align}\label{expch}
\pi^{-d/4}\Big(\sum_{\tilde{\mathbf m} \in H^\alpha(\mathbf m)} \det\mathcal W_{\tilde{\mathbf m}}^{-1/2}\Big)^{-1/2}\qquad \text{resp.} \quad \pi^{-d/4}\Big(\sum_{\tilde{\mathbf m} \in  H(\underline{\mathbf m})} \det\mathcal W_{\tilde{\mathbf m}}^{-1/2}\Big)^{-1/2}
\end{align}
with $H^\alpha(\mathbf m)=\{\tilde{\mathbf m}\in \mathcal U^{(0)}\cap E^\alpha(\mathbf m)\, ; \, W(\tilde{\mathbf m})=W(\mathbf m)\}$, $ H(\underline{\mathbf m})=\{\tilde{\mathbf m}\in \mathcal U^{(0)}\, ; \, W(\tilde{\mathbf m})=W(\underline{\mathbf m})\}$ and
\begin{align}\label{hessiennes}
\mathcal W_x\text{ is the Hessian of } W \text{ at }x.
\end{align}
Let us introduce the coefficients that we will use to define our quasimodes, in the spirit of \cite{michel}.
We denote $\mathcal F_\alpha$ the finite-dimensional vector space of functions from $\widehat{\mathcal U}_\alpha^{(0)}$ into $\R$ endowed with the natural euclidian structure
$$\langle \varphi , \varphi' \rangle_{\mathcal F_\alpha}=\sum_{\tilde{\mathbf m}\in\widehat{\mathcal U}_\alpha^{(0)}}\varphi (\tilde{\mathbf m}) \varphi'(\tilde{\mathbf m}).$$
Denoting also
\begin{align}\label{u01}
\mathcal U^{(0),I}_\alpha \text{ the elements of } \mathcal U^{(0)}_\alpha \text{ of type I}
\end{align}
and using \eqref{expch}, the following is established in \cite{michel}, Lemma 3.6 and below.
\begin{lem}\label{phiexist}
Recall the notation \eqref{hessiennes}.
One can construct an $h$-dependent orthonormal family $(\varphi^\alpha_{\mathbf m})_{\mathbf m \in \widehat{\mathcal U}_\alpha^{(0)}}\subset \mathcal F_\alpha$ such that
\begin{enumerate}[label=\alph*)]
\item $\varphi^\alpha_{\widehat{\mathbf m}}(\tilde{\mathbf m})=\widehat{c}_h c^\alpha_h(\tilde{\mathbf m})^{-1} \1_{\widehat{\mathcal U}^{(0)}_\alpha \backslash \mathcal U^{(0),I}_\alpha }(\tilde{\mathbf m})$ where $\widehat{c}_h$ is a normalization constant such that $\|\varphi^\alpha_{\widehat{\mathbf m}}\|_{\mathcal F_\alpha}=1$. \label{phimchapeau}
\item If $\{\mathbf m, \tilde{\mathbf m}\}\cap \mathcal U^{(0),I}_\alpha \neq \emptyset $, then $\varphi^\alpha_{\mathbf m}(\tilde{\mathbf m})=\delta_{\mathbf m,\tilde{\mathbf m}}$.\label{phidelta}
\item If $\varphi^\alpha_{\mathbf m}(\tilde{\mathbf m})\neq 0$, then $W(\mathbf m)=W(\tilde{\mathbf m})$.\label{phinonnul}
\item Each $\varphi^\alpha_{\mathbf m}$ admits an asymptotic expansion.
The leading term of $\varphi^\alpha_{\widehat{\mathbf m}}$ is given by
	$$\varphi^{\alpha,0}_{\widehat{\mathbf m}}(\tilde{\mathbf m})=\widehat{c}_0 c^\alpha_0(\tilde{\mathbf m})^{-1} \1_{\widehat{\mathcal U}^{(0)}_\alpha \backslash \mathcal U^{(0),I}_\alpha }(\tilde{\mathbf m})=\Bigg(\frac{\sum_{\mathring{\mathbf m} \in H^\alpha(\tilde{\mathbf m})} \det\mathcal W_{\mathring{\mathbf m}}^{-1/2}}{\sum_{\mathbf m' \in \widehat{\mathcal U}^{(0)}_\alpha \backslash \mathcal U^{(0),I}}\sum_{\mathring{\mathbf m} \in H^\alpha(\mathbf m')} \det\mathcal W_{\mathring{\mathbf m}}^{-1/2}}\Bigg)^{1/2} \1_{\widehat{\mathcal U}^{(0)}_\alpha \backslash \mathcal U^{(0),I}_\alpha }(\tilde{\mathbf m}).$$ 
Finally, for all $\mathbf m \in \mathcal U_\alpha^{(0)}$, the leading term of $\varphi^\alpha_{\mathbf m}$ can be computed explicitly and is orthogonal to the one of $\varphi^\alpha_{\widehat{\mathbf m}}$. \label{expphi}
\end{enumerate}
\end{lem}
\hip
For $\mathbf m \in \mathcal U_\alpha^{(0)}$, our quasimodes will be the functions 
\begin{align}\label{quasim}
f_{\mathbf m,h}(x)=(a^h)^{-1}h^{-d/4}\chi_\alpha(x)\bigg(\sum_{\tilde{\mathbf m}\in \widehat{\mathcal U}_\alpha^{(0)}} \varphi^\alpha_{\mathbf m}(\tilde{\mathbf m}) c^\alpha_h(\tilde{\mathbf m}) \theta^\alpha_{\tilde{\mathbf m},h}(x)\bigg)  \e^{-W_{\mathbf m}(x)/h}.
\end{align}
Note that $f_{\mathbf m,h}$ belongs to $\mathcal C^\infty_c(\R^{d})$ thanks to item \ref{hypoln} from Hypothesis \ref{hypol} and that
\begin{align}\label{suppf}
\mathrm{supp}\,f_{\mathbf m,h}\subseteq E_-(\mathbf m)
\end{align}
thanks to \eqref{supptheta}.

\section{Orthogonality}

The goal of this section is to show that the family of quasimodes that we introduced in \eqref{quasim} and \eqref{noyau} is almost orthonormal.
This result was already established in \cite{BonyLPMichel} in the case where $W$ has no type II minimum (see Remark 6.3 from \cite{BonyLPMichel}). Therefore, we will consider $\mathbf m\in \mathcal U_\alpha^{(0)}$, $\mathbf m'\in \mathcal U_{\alpha'}^{(0)}$ and we will study here the orthogonality of the quasimodes $f_{\mathbf m,h}$ and $f_{\mathbf m',h}$ with $\mathbf m$ or $\mathbf m'$ (or both) a type II minimum.
We follow the spirit of \cite{michel} (Proposition 3.10) and adapt it to the gaussian quasimodes framework.
\hop
\textbullet The case where $\mathbf m \mathcal R \mathbf m'$ and one of them is of type I, say $\mathbf m$ (in particular $\mathbf m \neq \mathbf m'$ because of our assumption).
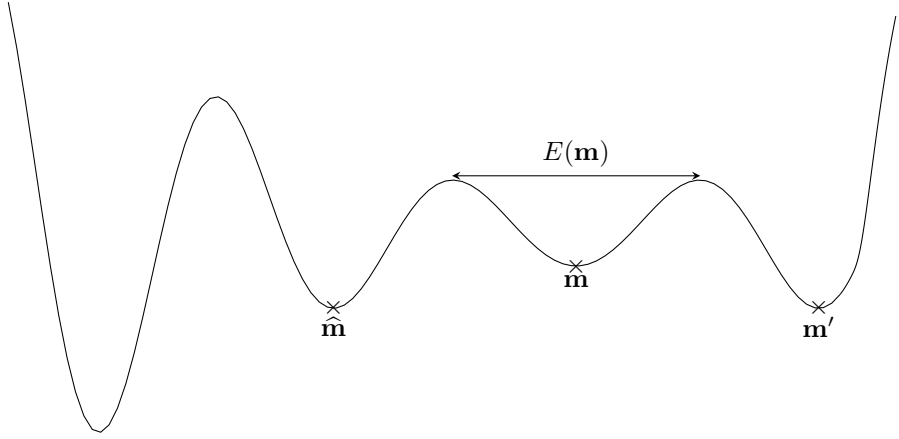
\begin{center}
\begin{tikzpicture}
\begin{scope} [scale=0.5]
\draw [samples=100,domain=-15.1:7.02] plot (\x,{-cos(\x r)*(1+(\x/6)^2)});
\draw [samples=100,domain=7.02:8.5] plot (\x,{-cos(\x r)*(1+(\x/6)^2)+6*exp(-1/(\x-7)^2)});
\draw (0,-1) node[below] {$\mathbf m$};
\draw (0,-1) node {$\times$};
\draw (-6.45,-2.1) node {$\times$};
\draw (-6.45,-2.1) node[below] {$\widehat{\mathbf m}$};
\draw (6.45,-2.1) node {$\times$};
\draw (6.45,-2.1) node[below] {$\mathbf m'$};
\draw [>=stealth, <->] (-3.3,1.4)--(3.3,1.4);
\draw (0,1.4) node[above] {$E(\mathbf m)$};
\end{scope}
\end{tikzpicture}
\end{center}
\hip
In that case, item \ref{phidelta} from Lemma \ref{phiexist} implies $\varphi_{\mathbf m}(\tilde{\mathbf m})=\delta_{\mathbf m, \tilde{\mathbf m}}$ for all $\tilde{\mathbf m}\in \widehat{\mathcal U}_\alpha^{(0)}$ and $\varphi_{\mathbf m'}(\mathbf m)= 0$ so by \eqref{supptheta}, we have
$$\mathrm{supp} \, f_{\mathbf m,h}\subseteq E(\mathbf m)+B(0,\varepsilon ') \qquad \text{and} \qquad \mathrm{supp} \,f_{\mathbf m',h}\subseteq \big(\R^d\backslash E(\mathbf m)\big)+B(0,\varepsilon ')$$
and hence
$$\mathrm{supp} \,f_{\mathbf m,h} \cap \mathrm{supp} \,f_{\mathbf m',h}\subseteq \big\{2W-W(\mathbf m)-W(\mathbf m')\geq c>0\big\}.$$
Consequently, $\langle f_{\mathbf m,h},  f_{\mathbf m',h} \rangle = O(\e^{-c/h})$.
\hop
\textbullet The case where $\mathbf m \mathcal R \mathbf m'$ (i.e $\alpha=\alpha'$) and both minima are of type II. 
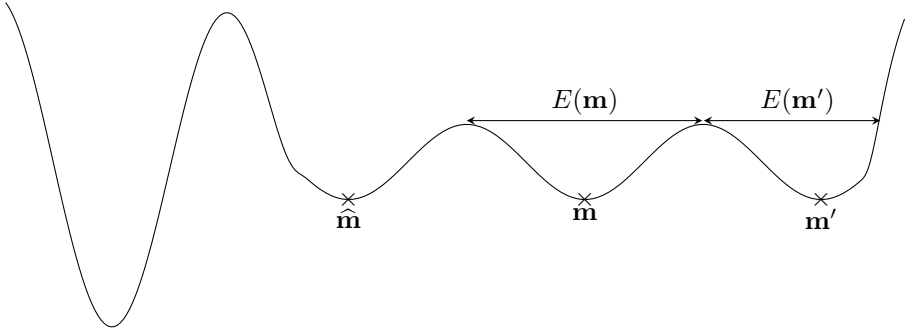
\begin{center}
\begin{tikzpicture}
\begin{scope} [scale=0.5]
\draw [samples=100,domain=-7.02:7.02] plot (\x,{-cos(\x r)});
\draw [samples=100,domain=-15.4:-7.02] plot (\x,{-cos(\x r)*(1+3.5*exp(-1/(\x+7)^2))});
\draw [samples=100,domain=7.02:8.5] plot (\x,{-cos(\x r)+5*exp(-1/(\x-7)^2)});
\draw (0,-1) node[below] {$\mathbf m$};
\draw (0,-1) node {$\times$};
\draw (2*pi,-1) node[below] {$\mathbf m'$};
\draw (2*pi,-1) node {$\times$};
\draw (-2*pi,-1) node[below] {$\widehat{\mathbf m}$};
\draw (-2*pi,-1) node {$\times$};
\draw [>=stealth, <->] (-pi,1.1)--(pi-0.02,1.1);
\draw (0,1) node[above] {$E(\mathbf m)$};
\draw [>=stealth, <->] (pi+0.02,1.1)--(2.5*pi,1.1);
\draw (1.8*pi,1) node[above] {$E(\mathbf m')$};
\end{scope}
\end{tikzpicture}
\end{center}
\hip
In that case, start by noticing that because of \eqref{supptheta}, if $\tilde{\mathbf m}$, $\tilde{\mathbf m}' \in \widehat{\mathcal U}_\alpha^{(0)}$ with $\tilde{\mathbf m}\neq \tilde{\mathbf m}'$, we have
$$\mathrm{supp}\, \theta^\alpha_{\tilde{\mathbf m},h} \cap \mathrm{supp}\, \theta^\alpha_{\tilde{\mathbf m}',h}\subseteq \{W_{\tilde{\mathbf m}}\geq c>0\}.$$
Therefore, using the definitions of our quasimodes \eqref{quasim} as well as item \ref{phinonnul} from Lemma \ref{phiexist} and \eqref{renorm}, we compute
\begin{align}
\langle f_{\mathbf m,h} , f_{\mathbf m',h} \rangle +O(\e^{-c/h})&=(a^h)^{-2}h^{-d/2} \sum_{\tilde{\mathbf m}\in \widehat{\mathcal U}_\alpha^{(0)}} \varphi^\alpha_{\mathbf m}(\tilde{\mathbf m}) \varphi^\alpha_{\mathbf m'}(\tilde{\mathbf m}) c^\alpha_h(\tilde{\mathbf m})^2 \int_{\R^d} \chi_\alpha^2 (\theta^\alpha_{\tilde{\mathbf m},h})^2 \e^{-2W_{\tilde{\mathbf m}}/h} \D x\\
		&=\langle \varphi^\alpha_{\mathbf m} , \varphi^\alpha_{\mathbf m'} \rangle_{\mathcal F_\alpha} \\
		&=\delta_{\mathbf m, \mathbf m'}
\end{align}
by Lemma \ref{phiexist}.
\hop
\textbullet The case where $\mathbf m \cancel{\mathcal R} \mathbf m'$ and $\boldsymbol \sigma (\mathbf m)=\boldsymbol \sigma (\mathbf m')$.
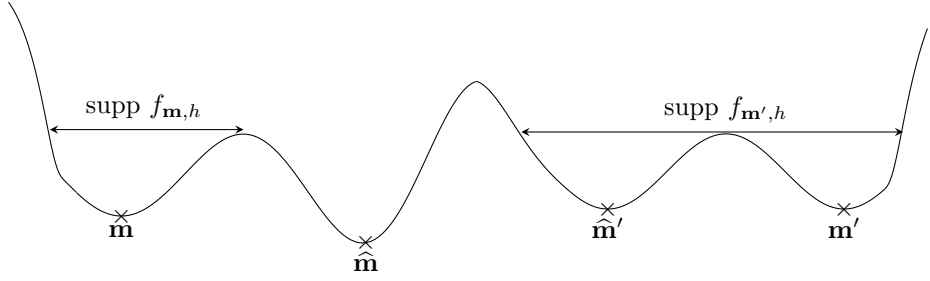
\begin{center}
\begin{tikzpicture}
\begin{scope} [scale=0.5]
\begin{scope} [yshift=-5.3,xshift=-368]
\draw [samples=100,domain=-1.2:3*pi] plot (\x,{-cos(\x r)*(1+(\x/7.5)^2)});
\draw [samples=100,domain=-3:-1.2] plot (\x,{-cos(\x r)+5*exp(-1/(\x+1.18)^2)});
\draw (0,-1) node[below] {$\mathbf m$};
\draw (0,-1) node {$\times$};
\draw (6.5,-1.7) node[below] {$\widehat{\mathbf m}$};
\draw (6.5,-1.7) node {$\times$};
\draw [>=stealth,<->] (-1.9,1.3)--(3.25,1.3);
\draw (0.6,1.3) node[above] {supp $f_{\mathbf m,h}$};
\end{scope}
\draw [samples=100,domain=-1:7.02] plot (\x,{-cos(\x r)});
\draw [samples=100,domain=-3.5:-1] plot (\x,{-cos(\x r)*(1+(\x/2+1/2)^2)});
\draw [samples=100,domain=7.02:8.5] plot (\x,{-cos(\x r)+5*exp(-1/(\x-7)^2)});
\draw (0,-1) node[below] {$\widehat{\mathbf m}'$};
\draw (0,-1) node {$\times$};
\draw (2*pi,-1) node[below] {$\mathbf m'$};
\draw (2*pi,-1) node {$\times$};
\draw [>=stealth,<->] (-2.3,1.05)--(2.5*pi,1.05);
\draw (pi,1.05) node[above] {supp $f_{\mathbf m',h}$};
\end{scope}
\end{tikzpicture}
\end{center}
\hip
In that case, it follows from the support properties of our quasimodes and the result of Proposition 3.8 from \cite{michel} that $\varepsilon'$ can be chosen small enough so that the supports of $f_{\mathbf m,h}$ and $f_{\mathbf m',h}$ do not intersect.
Thus, $\langle f_{\mathbf m,h},  f_{\mathbf m',h} \rangle = 0$.
\hop
\textbullet The case where $\mathbf m \cancel{\mathcal R} \mathbf m'$, $\boldsymbol \sigma (\mathbf m)> \boldsymbol \sigma (\mathbf m')$ and $W(\mathbf m)=W(\mathbf m')$.
\begin{center}
\begin{tikzpicture}
\begin{scope} [scale=0.5]
\draw [samples=100,domain=-7.02:7.02] plot (\x,{-cos(\x r)});
\draw [samples=100,domain=-15.4:-7.02] plot (\x,{-cos(\x r)*(1+3.5*exp(-1/(\x+7)^2))});
\draw [samples=100,domain=7.02:8.5] plot (\x,{-cos(\x r)+5.5*exp(-1/(\x-7)^2)});
\draw (0,-1) node[below] {$\mathbf m'$};
\draw (0,-1) node {$\times$};
\draw (-2*pi,-1) node[below] {$\mathbf m=\widehat{\mathbf m}'$};
\draw (-2*pi,-1) node {$\times$};
\draw [>=stealth, <->] (-pi,1.1)--(pi-0.02,1.1);
\draw (0,1) node[above] {$E(\mathbf m')$};
\draw [>=stealth, <->] (-3*pi,4.1)--(2.65*pi,4.1);
\draw (0,4.1) node[above] {$E(\mathbf m)$};
\end{scope}
\end{tikzpicture}
\end{center}
\hip
In that situation, thanks to \eqref{quasim}, \eqref{suppf} and \eqref{supptheta} we can suppose that
\begin{align}\label{fmsuppfm'}
f_{\mathbf m,h}=(a^h)^{-1}\tilde c_h(\mathbf m)h^{-d/4}\e^{-W_{\mathbf m}/h} \qquad \text{on the support of } f_{\mathbf m',h}, \text{ with } \tilde c_h(\mathbf m)=O(1)
\end{align}
(otherwise the supports of $f_{\mathbf m,h}$ and $f_{\mathbf m',h}$ are disjoint and the result is obvious) and consequently $\mathbf m'$ is of type II.
Hence, using a standard Laplace method, we can write 
\begin{align}
\langle f_{\mathbf m,h},  f_{\mathbf m',h} \rangle &= (a^h)^{-2} h^{-d/2} \tilde c_h(\mathbf m) \sum_{\tilde{\mathbf m}'\in \widehat{\mathcal U}_{\alpha'}^{(0)}} \varphi^{\alpha'}_{\mathbf m'}(\tilde{\mathbf m}') c^{\alpha'}_h(\tilde{\mathbf m}') \int_{\R^d} \chi_{\alpha'} \theta^{\alpha'}_{\tilde{\mathbf m}',h}\, \e^{-2W_{\mathbf m}/h} \D x\\
		&= (a^h)^{-2} h^{-d/2} \tilde c_h(\mathbf m) \sum_{\tilde{\mathbf m}'\in \widehat{\mathcal U}_{\alpha'}^{(0)}} \varphi^{\alpha'}_{\mathbf m'}(\tilde{\mathbf m}') c^{\alpha'}_h(\tilde{\mathbf m}') \int_{\R^d} \chi_{\alpha'}^2 (\theta^{\alpha'}_{\tilde{\mathbf m}',h})^2\, \e^{-2W_{\mathbf m}/h} \D x+O(\e^{-c/h})\\
		&= \tilde c_h(\mathbf m) \sum_{\tilde{\mathbf m}'\in \widehat{\mathcal U}_{\alpha'}^{(0)}} \varphi^{\alpha'}_{\mathbf m'}(\tilde{\mathbf m}') c^{\alpha'}_h(\tilde{\mathbf m}')^{-1} +O(\e^{-c/h})\qquad \text{by \eqref{renorm}} \\
		&= \tilde c_h(\mathbf m) \sum_{\tilde{\mathbf m}'\in \widehat{\mathcal U}_{\alpha'}^{(0)}\backslash \mathcal U^{(0),I}_{\alpha'}} \varphi^{\alpha'}_{\mathbf m'}(\tilde{\mathbf m}') c^{\alpha'}_h(\tilde{\mathbf m}')^{-1} +O(\e^{-c/h})\qquad \text{by item \ref{phidelta} from Lemma \ref{phiexist}}\\
		&=\frac{\tilde c_h(\mathbf m)}{\widehat c_{h}} \langle \varphi^{\alpha'}_{\mathbf m'}, \varphi^{\alpha'}_{\widehat{\mathbf m}'} \rangle_{\mathcal F_{\alpha'}} +O(\e^{-c/h}) \qquad \text{by item \ref{phimchapeau} from Lemma \ref{phiexist}} \\
		&=O(\e^{-c/h})
\end{align}
where we also used the orthogonality of the family $(\varphi_{\mathbf m})_{\mathbf m\in \widehat{\mathcal U}_{\alpha'}^{(0)}}$.
\hop
\textbullet The case where $\mathbf m \cancel{\mathcal R} \mathbf m'$, $\boldsymbol \sigma (\mathbf m)> \boldsymbol \sigma (\mathbf m')$ and $W(\mathbf m)\neq W(\mathbf m')$.
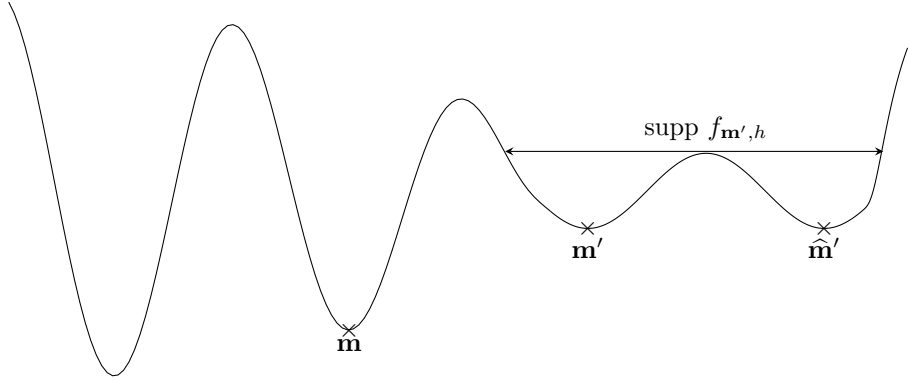
\begin{center}
\begin{tikzpicture}
\begin{scope} [scale=0.5]
\draw [samples=100,domain=-1:7.02] plot (\x,{-cos(\x r)});
\draw [samples=100,domain=-15.4:-1] plot (\x,{-cos(\x r)*(1+0.8*ln(1+(1+\x)^2))});
\draw [samples=100,domain=7.02:8.5] plot (\x,{-cos(\x r)+5*exp(-1/(\x-7)^2)});
\draw (0,-1) node[below] {$\mathbf m'$};
\draw (0,-1) node {$\times$};
\draw (2*pi,-1) node[below] {$\widehat{\mathbf m}'$};
\draw (2*pi,-1) node {$\times$};
\draw [>=stealth, <->] (-2.2,1.05)--(2.5*pi,1.05);
\draw (pi,1.05) node[above] {supp $f_{\mathbf m',h}$};
\draw (-6.35,-3.7) node {$\times$};
\draw (-6.35,-3.7) node[below] {$\mathbf m$};
\end{scope}
\end{tikzpicture}
\end{center}
\hip
Here again we can suppose that 
\eqref{fmsuppfm'} holds true so 
$W(\mathbf m)<W(\mathbf m')$.
By item \ref{phinonnul} from Lemma \ref{phiexist}, we have $W\geq W(\mathbf m')$ on the support of $f_{\mathbf m',h}$.
Therefore, $W_{\mathbf m}\geq c>0$ on the support of $f_{\mathbf m',h}$ and since $f_{\mathbf m',h}=O_{L^\infty}(h^{-d/4})$, we get $\langle f_{\mathbf m,h},  f_{\mathbf m',h} \rangle = O(\e^{-c/h})$.
\hop
As a result of the above discussion, we obtain the following statement.
\begin{prop}\label{ortho}
The family of quasimodes $(f_{\mathbf m,h})_{\mathbf m\in \mathcal U^{(0)}}$ introduced in \eqref{quasim} and \eqref{noyau} is almost orthonormal:
$$\langle f_{\mathbf m,h},  f_{\mathbf m',h} \rangle = \delta_{\mathbf m, \mathbf m'} +O(\e^{-c/h}).$$
\end{prop}

\section{Action of the operator $P_h$}

Let us fix $\mathbf m \in \mathcal U^{(0)}_\alpha$. 
For $\tilde{\mathbf m}\in \widehat{\mathcal U}^{(0)}_\alpha$ we will denote
\begin{align}\label{ytilde}
\widetilde W_{\tilde{\mathbf m},h}=W_{\mathbf m}+\sum_{\mathbf s \in \mathbf j^\alpha(\tilde{\mathbf m})}(\ell^{\mathbf s,\tilde{\mathbf m}})^2/2
\end{align}
and
\begin{align}\label{psi}
\psi^{\tilde{\mathbf m},h}(x,y)=\int_0^1\nabla \widetilde W_{\tilde{\mathbf m},h}( y+t(x-y)) \D t.
\end{align}

\begin{rema}\label{defg}
Using Hypotheses \ref{W} and \ref{hyporw} as well as symbolic calculus, one gets 
$d_W^*\text{Op}_h(q^h)=\text{Op}_h(g^h)$, with $g^h\sim \sum_{n}h^ng_n$ in $\mathcal M_{1,d}\big(S^0_\kappa(\langle x \rangle^\eta)\big)$ given by
\begin{align}\label{g0}
g_0(x,\xi)=\Big(-i\, \xi^t+ \, \nabla W^t\Big)q_0(x, \xi)
\end{align}
and
\begin{align}\label{gn}
g_n(x,\xi)=\Big(-i\, \xi^t+ \, \nabla W^t\Big)q_n(x, \xi)- \nabla_x^t q_{n-1}(x,\xi)  +\sum_{k=0}^n i^k \mathop{\sum_{\beta \in \N^d;}}_{\substack{|\beta|=k}}  c_{k,\beta}(x) \partial_\xi^\beta(q_{n-k})(x,\xi)
\end{align}
for some $c_{k,\beta}\in S^0(\langle x\rangle^\eta)$ taking values in $\R^d$.
\end{rema}

\begin{lem}\label{gborne}
The operator $\mathrm{Op}_h (g^h)=d_W^*\circ \widehat Q$ introduced in Remark \ref{defg} is bounded on $L^2(\R^d)$.
\end{lem}

\begin{proof}
Since $\widehat Q$ is self-adjoint, it is sufficient to prove that $\widehat Q\circ d_W$ is bounded.
Thanks to the facts that $\widehat Q$ is bounded and non negative, we can simply write for $u\in L^2(\R^d)$
\begin{align}
\|\widehat Qd_W u\|^2&\leq C\|\widehat Q^{1/2}d_Wu\|^2\\
	&\leq C \langle \widehat Qd_Wu, d_Wu \rangle\\
	&\leq C \langle P_hu,u \rangle\\
	&\leq C \|u\|^2 
\end{align}
according to Hypothesis \ref{hyporw}, and the statement is proven.
\end{proof}

\begin{prop}\label{phf}
Let $f_{\mathbf m,h}$ be the quasimode defined in \eqref{quasim}.
With the notations introduced in \eqref{approxa} and \eqref{ytilde}, one has
$$P_hf_{\mathbf m,h}=a^h\frac{h^{1-d/4}}{2} A_h^{-1}\sum_{\tilde{\mathbf m}\in \widehat{\mathcal U}_\alpha^{(0)}} \varphi^\alpha_{\mathbf m}(\tilde{\mathbf m}) c^\alpha_h(\tilde{\mathbf m})  \om^{\tilde{\mathbf m},\alpha}\,\e^{\frac{-\widetilde W_{\tilde{\mathbf m},h}}{h}}\1_{\mathbf j^\alpha(\tilde{\mathbf m})+B(0,2r)}+O_{L^2}\Big(h^\infty\e^{-\frac{S(\mathbf m)}{h}}\Big)$$
where $\om^{\tilde{\mathbf m},\alpha}$ is a function bounded uniformly in $h$ and defined for $x\in \mathbf j^\alpha(\tilde{\mathbf m})+B(0,2r)$ by
\begin{align*}
\om^{\tilde{\mathbf m},\alpha}(x)=&\sum_{ \mathbf s \in \mathbf j^\alpha(\tilde{\mathbf m})} (2\pi h)^{-d}\int_{\R^d} \int_{|y-\mathbf s|\leq 2r} \e^{\frac ih \xi \cdot (x-y)}g^h\Big(\frac{x+y}{2},\xi+i\psi^{\tilde{\mathbf m},h}(x,y)\Big) \nabla\ell^{\mathbf s,\tilde{\mathbf m}}(y)\,\D y \D \xi.
\end{align*}
\end{prop}

\begin{proof}
In order to lighten the notations, we will drop some of the exponents and indexes $\mathbf m$, $\mathbf s$, $\alpha$ and $h$ in the proof.
Let $\tilde{\mathbf m}\in \widehat{\mathcal U}_\alpha^{(0)}$.
By \eqref{hypol}, we have on the support of $\chi$ that $\theta_{\tilde{\mathbf m}}^\alpha$ is smooth and since it is constant outside of $B(\mathbf s,r)$, we have
\begin{align}\label{nablatheta}
\nabla \theta_{\tilde{\mathbf m}}^\alpha=\frac{A_h^{-1}}{2}\sum_{\mathbf s\in \mathbf j^\alpha(\tilde{\mathbf m})}\e^{-(\ell^{\mathbf s,\tilde{\mathbf m}})^2/2h}\zeta(\ell^{\mathbf s,\tilde{\mathbf m}})\nabla \ell^{\mathbf s,\tilde{\mathbf m}} \,\1_{B(\mathbf s,r)}.
\end{align}
We can then use Remark \ref{defg} to write
\begin{align}\label{qf}
  P_h(f)&=a^hh^{1-d/4} \sum_{\tilde{\mathbf m}\in \widehat{\mathcal U}_\alpha^{(0)}} \varphi^\alpha_{\mathbf m}(\tilde{\mathbf m}) c^\alpha_h(\tilde{\mathbf m}) \mathrm{Op}_h(g)\big(\nabla \theta_{\tilde{\mathbf m}}^\alpha \chi \e^{-W_{\mathbf m}/h}+\nabla\chi \theta_{\tilde{\mathbf m}}^\alpha \e^{-W_{\mathbf m}/h}\big)\nonumber\\
			&=a^h\frac{h^{1-d/4}}{2} A_h^{-1} \sum_{\tilde{\mathbf m}\in \widehat{\mathcal U}_\alpha^{(0)}} \varphi^\alpha_{\mathbf m}(\tilde{\mathbf m}) c^\alpha_h(\tilde{\mathbf m}) \sum_{\mathbf s\in \mathbf j^\alpha(\tilde{\mathbf m})} \mathrm{Op}_h(g)\Big(\zeta(\ell^{\mathbf s,\tilde{\mathbf m}}) \chi \e^{\frac{-\widetilde W_{\tilde{\mathbf m},h}}{h}} \nabla \ell^{\mathbf s,\tilde{\mathbf m}} \,\1_{B(\mathbf s,r)}\Big)+O\Big(h\e^{-\frac{S(\mathbf m)+\tilde \varepsilon}{h}}\Big) 
\end{align}
where we used \eqref{suppchi} and Lemma \ref{gborne}.
Now we have for $\mathbf s\in \mathbf j^\alpha(\tilde{\mathbf m})$
\begin{align}\label{pasxi}
(2 \pi h)^{d}\mathrm{Op}_h(g)\Big(\zeta(\ell^{\mathbf s,\tilde{\mathbf m}}) \chi \e^{\frac{-\widetilde W_{\tilde{\mathbf m},h}}{h}}& \nabla \ell^{\mathbf s,\tilde{\mathbf m}} \,\1_{B(\mathbf s,r)}\Big)(x)=\int_{\R^{d}} \int_{|y-\mathbf s|\leq r} \e^{\frac ih \xi \cdot (x-y)}g\Big(\frac{x+y}{2},\xi\Big) \\
& \qquad \qquad \qquad \quad \times\chi(y) \zeta\big(\ell^{\mathbf s,\tilde{\mathbf m}}(y)\big) \e^{-\widetilde W_{\tilde{\mathbf m}}(y)/h}\nabla\ell^{\mathbf s,\tilde{\mathbf m}}(y)\;\D y \D \xi . \nonumber
\end{align}
Let us now treat separately the cases $|x-\mathbf s|\geq 2r$ and $|x-\mathbf s|<2r$ .\\
When $|x-\mathbf s|\geq 2r$, we have $|x-y|\geq r$ so we can apply the non stationnary phase to the integral in $\xi$ to get that for all $N\geq 1$, there exists $C_N>0$ such that
\begin{align*}
\bigg|\int_{\R^{d}} \int_{|y-\mathbf s|\leq r} \e^{\frac ih \xi \cdot (x-y)}g\Big(\frac{x+y}{2},\xi\Big) \chi(y) \zeta\big(\ell^{\mathbf s,\tilde{\mathbf m}}(y)\big) \e^{-\widetilde W_{\tilde{\mathbf m}}(y)/h}\nabla\ell^{\mathbf s,\tilde{\mathbf m}}(y)\;\D y \D \xi \bigg| \leq C_N h^{N} |x-\mathbf s|^{\eta-N} \e^{-\frac{S(\mathbf m)}{h}}
\end{align*}
where we used item \ref{hypoln-1} from \eqref{hypol}, the fact that $W_{\mathbf m}(\mathbf s)+(\ell_0^{\mathbf s,\tilde{\mathbf m}})^2(\mathbf s)/2=S(\mathbf m)$ and the estimate $|x-y|\geq |x-\mathbf s|/2$.
Hence we have shown that 
\begin{align}\label{vgrand}
 P_hf\,\1_{\big\{\mathrm{dist}\big(\cdot,\cup_{\tilde{\mathbf m}\in \widehat{\mathcal U}^{(0)}_\alpha}\mathbf j^\alpha(\tilde{\mathbf m})\big)\geq 2r\big\}}=O\Big(h^\infty \e^{-\frac{S(\mathbf m)}{h}}\Big).
\end{align}
Now for the case $|x-\mathbf s|<2r$, let us denote $J_1^{\mathbf s,\tilde{\mathbf m}}(x)$ the RHS of \eqref{pasxi}. Proceeding as in \cite{Nakamura} in order to take the $\e^{-\widetilde W_{\tilde{\mathbf m}}(y)/h}$ in front of the oscillatory integral, 
we get that
\begin{align}\label{expdevant}
J_1^{\mathbf s,\tilde{\mathbf m}}(x)=\e^{-\widetilde W_{\tilde{\mathbf m}}(x)/h}J_2^{\mathbf s,\tilde{\mathbf m}}(x)
\end{align}
where
\begin{align}
J_2^{\mathbf s,\tilde{\mathbf m}}(x)=\int_{\R^d} \int_{|y-\mathbf s|\leq r} \e^{\frac ih \big(\xi-i\psi(x,y)\big) \cdot \big(x-y\big)}g\Big(\frac{x+y}{2},\xi\Big) \chi(y) \zeta\big(\ell^{\mathbf s,\tilde{\mathbf m}}(y)\big) \nabla\ell^{\mathbf s,\tilde{\mathbf m}}(y)\;\D y \D \xi 
\end{align}
and $\psi$ is the function defined in \eqref{psi}.
Note that up to taking $r$ smaller, we can suppose that $|\psi|<1$ (and thus $\xi +i\psi$ lies in the analyticity strip of $g$) on $B(\mathbf s,2r)\times \{|y-\mathbf s|<r\}$.
Applying the Cauchy formula as in \cite{me} (proof of Proposition 3.13), one gets $J_2^{\mathbf s,\tilde{\mathbf m}}=J_3^{\mathbf s,\tilde{\mathbf m}}$ where
\begin{align}
J_3^{\mathbf s,\tilde{\mathbf m}}(x)=\int_{\R^d} \int_{|y-\mathbf s|\leq r} \e^{\frac ih \xi \cdot (x-y)}g\Big(\frac{x+y}{2},\xi+i\psi(x,y)\Big)\chi(y) \zeta\big(\ell^{\mathbf s,\tilde{\mathbf m}}(y)\big) \nabla\ell^{\mathbf s,\tilde{\mathbf m}}(y)\;\D y \D \xi 
\end{align}
Combined with 
\eqref{pasxi} and \eqref{expdevant}, this yields for $|x-\mathbf s|<2r$
\begin{align}\label{opgtruc}
(2 \pi h)^{d}\mathrm{Op}_h(g)\Big(\zeta(\ell^{\mathbf s,\tilde{\mathbf m}}) \chi \e^{\frac{-\widetilde W_{\tilde{\mathbf m},h}}{h}}& \nabla \ell^{\mathbf s,\tilde{\mathbf m}} \,\1_{B(\mathbf s,r)}\Big)(x)=\e^{\frac{-\widetilde W_{\tilde{\mathbf m},h}(x)}{h}}J_3^{\mathbf s,\tilde{\mathbf m}}(x).
\end{align}
Therefore, setting on $\mathbf j^\alpha(\tilde{\mathbf m})+B(0,2r)$
\begin{align*}
\tilde \om^{\tilde{\mathbf m},\alpha}= (2\pi h)^{-d} \sum_{\mathbf s\in \mathbf j^\alpha(\tilde{\mathbf m})}J_3^{\mathbf s,\tilde{\mathbf m}}(x),
\end{align*}
we have according to \eqref{qf}, \eqref{vgrand} and \eqref{opgtruc}
$$P_hf=a^h\frac{h^{1-d/4}}{2} A_h^{-1}\sum_{\tilde{\mathbf m}\in \widehat{\mathcal U}_\alpha^{(0)}} \varphi^\alpha_{\mathbf m}(\tilde{\mathbf m}) c^\alpha_h(\tilde{\mathbf m})  \tilde\om^{\tilde{\mathbf m},\alpha}\,\e^{\frac{-\widetilde W_{\tilde{\mathbf m},h}}{h}}\1_{\mathbf j^\alpha(\tilde{\mathbf m})+B(0,2r)}+O\Big(h^\infty\e^{-\frac{S(\mathbf m)}{h}}\Big).$$
Hence it is sufficient to check that on $\mathbf j^\alpha(\tilde{\mathbf m})+B(0,2r)$
\begin{align*}
\big(\tilde\om^{\tilde{\mathbf m},\alpha}-\om^{\tilde{\mathbf m},\alpha} \big)\e^{\frac{-\widetilde W_{\tilde{\mathbf m},h}}{h}}=O\Big(h^\infty \e^{-\frac{S(\mathbf m)}{h}}\Big).
\end{align*}
This can be done easily using again the non stationary phase with $x$ in an $h$-independent neighborhood of $\mathbf s$ on which $\chi\zeta(\ell)-1$ vanishes since item \ref{hypoln-1} from \eqref{hypol} implies that 
$$\e^{\frac{-\widetilde W_{\tilde{\mathbf m},h}}{h}}=O(\e^{-(S(\mathbf m)+\delta)/h})$$
outside of this neighborhood for some $\delta>0$.
\end{proof}

\section{Choice of $\ell^{\mathbf s,\tilde{\mathbf m}}$}\label{sectionequations}
\noindent
From now on, we also fix $\tilde{\mathbf m}\in \widehat{\mathcal U}^{(0)}_\alpha$ and $\mathbf s\in \mathbf j^\alpha(\tilde{\mathbf m})$.
We write for shortness $\ell^{\mathbf s}$ instead of $\ell^{\mathbf s,\tilde{\mathbf m}}$.

\begin{lem}\label{expom}
The function $\om^{\tilde{\mathbf m},\alpha}$ admits the classical expansion $\om^{\tilde{\mathbf m},\alpha} \sim \sum_{j\geq 0}h^j\om^{\tilde{\mathbf m},\alpha}_j$ on $B(\mathbf s,2r)$
where 
\begin{align}\label{om0}
\om^{\tilde{\mathbf m},\alpha}_0=q_0\Big(x,i\big(\nabla W +\ell^{\mathbf s}_0 \, \nabla \ell^{\mathbf s}_0\big)\Big)\big(2\nabla W+\ell^{\mathbf s}_0 \nabla \ell^{\mathbf s}_0\big)\cdot \nabla\ell^{\mathbf s}_0
\end{align}
and for $j\geq 1$, 
\begin{align}\label{omj}
\qquad \om^{\tilde{\mathbf m},\alpha}_{j}=&\,2\, q_0\Big(x,i\big(\nabla W +\ell^{\mathbf s}_0 \, \nabla \ell^{\mathbf s}_0\big)\Big)( \nabla W+\ell^{\mathbf s}_0 \nabla \ell^{\mathbf s}_0)\cdot \nabla \ell^{\mathbf s}_j \\
			&+i\, \ell^{\mathbf s}_0 \big(2\, \nabla W^t+\ell^{\mathbf s}_0 \, ( \nabla \ell^{\mathbf s}_0)^t\big) D_\xi q_0\big(x,i(\nabla W+\ell^{\mathbf s}_0 \nabla \ell^{\mathbf s}_0)\big)\big( \nabla \ell^{\mathbf s}_j\big) \, \nabla \ell^{\mathbf s}_0 \nonumber\\
			&+q_0\Big(x,i\big(\nabla W +\ell^{\mathbf s}_0 \, \nabla \ell^{\mathbf s}_0\big)\Big) \nabla \ell^{\mathbf s}_0 \cdot \nabla \ell^{\mathbf s}_0 \, \ell^{\mathbf s}_j\nonumber\\
			&+i  \big(2\,\nabla W^t+\ell^{\mathbf s}_0 \, ( \nabla \ell^{\mathbf s}_0)^t\big) D_\xi q_0\big(x,i(\nabla W+\ell^{\mathbf s}_0 \nabla \ell^{\mathbf s}_0)\big)\big( \nabla \ell^{\mathbf s}_0\big) \, \nabla \ell^{\mathbf s}_0 \, \ell^{\mathbf s}_j \nonumber\\
			&+R_j(\ell^{\mathbf s}_0, \dots, \ell^{\mathbf s}_{j-1})\nonumber
\end{align}
where $R_j : \big( \mathcal C^\infty(B(\mathbf s,2r) )\big)^j \to \mathcal C^\infty(B(\mathbf s,2r) )$ and $D_\xi$ denotes the partial differential with respect to the variable $\xi$.
\end{lem}

\begin{proof} %
Once again, we drop some of the exponents and indexes $\tilde{\mathbf m}$, $\mathbf s$, $\alpha$ and $h$ in the proof.
\sloppy Denote $B_\infty(\mathbf s,2r)=\{(y,\xi )\in \R^{2d}\,;\,\max (|y-\mathbf s|,|\xi|)< 2r\}$.
We need to get an expansion of $g(x/2+y/2,\xi+i\psi(x,y))$ that we will then be able to combine with the stationnary phase to get an expansion of 
$\om$.
Let us start with an expansion of $\psi$ :
the expansion of $\ell$ yields
$$\nabla \widetilde W-\nabla W\sim \sum_{j\geq 0}h^j\sum_{k=0}^j\ell_k\nabla \ell_{j-k}\quad \text{on } B(\mathbf s,2r)$$
so using \eqref{psi}, we get 
\begin{align}\label{exppsi}
\psi \sim \sum_{j\geq 0}h^j\psi_j\quad \text{on } B(\mathbf s,2r)\times \{|y|\leq 2r\}
\end{align}
where 
\begin{align}\label{psi0}
\psi_0(x,y)=\int_0^1\big(\nabla W+\ell_0\nabla \ell_0\big)(y+t(x-y))\D t
\end{align}
and for $j \geq 1$, 
\begin{align}\label{psij}
\psi_j(x,y)=\int_0^1\sum_{k=0}^j\big(\ell_k\nabla \ell_{j-k}\big)(y+t(x-y))\D t.
\end{align}
Proceeding as in \cite{me} (proof of Lemma 4.1), we then get thanks to 
Remark \ref{defg} that 
\begin{align}\label{expg}
g\Big(\frac{x+y}{2},\xi +i\psi(x,y)\Big)\sim \sum_{j\geq 0} h^j \sum_{n=0}^j g_{n,j-n}(x,y,\xi)
\end{align}
on $B(\mathbf s,2r)\times B_\infty (\mathbf s,2r)$; with
\begin{align}\label{gn0}
g_{n,0}(x,y,\xi)=g_n\Big(\frac{x+y}{2},\xi +i\psi_0(x,y)\Big)
\end{align}
and for $j \geq 1$
\begin{align}\label{gnj}
g_{n,j}(x,y,\xi)=iD_\xi g_n\Big(\frac{x+y}{2},\xi+i\psi_0(x,y)\Big) \big(\psi_j(x,y)\big)+R^1_j(\ell_0, \dots, \ell_{j-1})
\end{align}
where $R^1_j : \big( \mathcal C^\infty(B(\mathbf s,2r) )\big)^j \to \mathcal C^\infty(B(\mathbf s,2r) )$.
Thus, using the expansion \eqref{expg} that we just got, the one of $\nabla \ell$, and the one for an oscillatory integral given by the stationnary phase (see for instance \cite{Zworski}, Theorem 3.17) as well as Proposition C.3 from \cite{me}, we finally get
\begin{align}\label{expinteosci}
\om \sim \sum_{j\geq 0} h^j \om_j \quad \text{on } B(\mathbf s,2r),
\end{align}
where
$$\om_j(x)=\sum_{n_1+n_2+n_3+n_4=j}\frac{1}{i^{n_1}n_1!}\big(\partial_{y} \cdot \partial_\xi \big)^{n_1} \Big( g_{n_2,n_3}(x,y,\xi) \nabla \ell_{n_4}(y) \Big) \Bigg| \mathop{}_{\substack{y=x \\  \xi =0}}.$$
We can already use \eqref{gn0} to deduce the expression of $\omega_0$ by noticing that according to \eqref{psi0}, $\psi_0(x,x)=\nabla W+\ell_0 \nabla \ell_0$.
For $j \geq 1$, the terms of $\om_j$ in which the function $\ell_j$ appears are obviously the one given by $n_4=j$, but also the one given by $n_3=j$ according to \eqref{gnj}.
Indeed, in that case, we have using \eqref{psij} that
\begin{align*}
g_{0,j}(x,x,0)=i\ell_0  D_\xi g_0\big(x,i(\nabla W+&\ell_0 \nabla \ell_0)\big)\big(\nabla \ell_j\big)\\
				&+iD_\xi g_0\big(x,i(\nabla W+\ell_0 \nabla \ell_0)\big)\big(\nabla \ell_0\big) \, \ell_j+R^2_j(\ell_0, \dots, \ell_{j-1})
\end{align*}
where $R^2_j : \big( \mathcal C^\infty(B(\mathbf s,2r) )\big)^j \to \mathcal C^\infty(B(\mathbf s,2r) )$.
We can now conclude as for any $X \in \R^d$,
\begin{align*}D_\xi g_0\big(x,i(\nabla W+\ell_0 \nabla \ell_0)\big)(X)=-i\,X^tq_0&\big(x,i(\nabla W+\ell_0 \nabla \ell_0)\big) \\
					&+\big(2\,\nabla W^t+\ell_0 \, ( \nabla \ell_0)^t\big) D_\xi q_0\big(x,i(\nabla W+\ell_0 \nabla \ell_0)\big)(X) 
\end{align*}
according to \eqref{g0}.
\end{proof}
\hip
Denote $(q_{m,p}^n)_{m,p}$ the entries of the matrix $q_n$ from Hypothesis \ref{hyporw}. 
Since we have for $X\in \R^d$ 
$$ D_\xi q_0\big(x,i(\nabla W+\ell_0 \nabla \ell_0)\big)\big( X\big)=\Big( \partial_{\xi}q^0_{m,p}\big(x,i(\nabla W+\ell_0 \nabla \ell_0)\big)\cdot X\Big)_{1\leq m,p\leq d}$$
we get by putting
\begin{align}\label{U}
\quad U(x)=q_0\Big(x,&i\big(\nabla W +\ell_0 \, \nabla \ell_0\big)\Big) \nabla\ell_0 \\
&+\sum_{1\leq m,p \leq d}\big( 2\partial{x_m}W +\ell_0\partial_{x_m} \ell_0 \big) i \partial_{\xi} q^0_{m,p}\Big(x,i\big(\nabla W +\ell_0 \, \nabla \ell_0\big)\Big)\partial_{x_p}\ell_0 \nonumber
\end{align}
that equation \eqref{omj} reads 
\begin{align*}
\omega_j= \bigg(
				q_0\Big(x,i\big(\nabla W +\ell_0 \, \nabla \ell_0\big)\Big)( 2\nabla W+\ell_0 \nabla \ell_0)+\ell_0\, U
				\bigg) \cdot \nabla \ell_j+ U \cdot \nabla \ell_0  \, \ell_j+R_j(\ell_0, \dots, \ell_{j-1}).
\end{align*}

\begin{lem}\label{ureelle}
Let $x$, $y \in B(\mathbf s,2r)$. 
For any $n \in \N$, $\beta \in \N^d$ and $1 \leq m,p \leq d$, we have
$$\partial^\beta_\xi q_{m,p}^n\Big(\frac{x+y}{2},i\psi_0^{\tilde{\mathbf m},h}(x,y)\Big)  \in i^{|\beta|}\R$$
and
$$ \partial^\beta_\xi g_n\Big(\frac{x+y}{2},i\psi_0^{\tilde{\mathbf m},h}(x,y)\Big)  \in i^{|\beta|}\R^d.$$
In particular, $U$ defined in \eqref{U} sends $B(\mathbf s,2r)$ in $\R^d$.
\end{lem}

\begin{proof} %
Since $\ell_0$ vanishes at $\mathbf s$, we can suppose that $r$ is such that $i\psi_0(x,y)$ is in 
\begin{align}\label{d0tau}
D(0, 1)^d=\{z\in \C \, ; \, |z|<1\}^d
\end{align}
so by analyticity and using the parity of $q_{m,p}^n$
, we have
\begin{align*}
\partial^\beta_\xi q_{m,p}^n\Big(\frac{x+y}{2},i\psi_0(x,y)\Big)=\mathop{\sum_{\gamma \in \N^d;}}_{\substack{|\gamma|+|\beta|\in 2\N}} i^{|\gamma|}\frac{\partial_\xi^{\gamma+\beta} q_{m,p}^n\big(\frac{x+y}{2},0\big)}{\gamma !}\,\psi_0(x,y)^\gamma \in i^{|\beta|}\R.
\end{align*}
The result for $g_n$ follows easily using \eqref{g0} and \eqref{gn}.
\end{proof}
\hip
Using this Lemma, we also get the following result (see \cite{me} Appendix D for a proof).

\begin{lem}\label{rreelle}
The term $R_j(\ell_0^{\mathbf s,\tilde{\mathbf m}}, \dots, \ell_{j-1}^{\mathbf s,\tilde{\mathbf m}})$ from Lemma \ref{expom} is real valued.
Moreover, it satisfies 
$$R_j(\ell_0^{\mathbf s,\tilde{\mathbf m}}, \dots, \ell_{j-1}^{\mathbf s,\tilde{\mathbf m}})=-R_j(-\ell_0^{\mathbf s,\tilde{\mathbf m}}, \dots, -\ell_{j-1}^{\mathbf s,\tilde{\mathbf m}}).$$
\end{lem}

\hip
In view of the results from Proposition $\ref{phf}$ and Lemma \ref{expom}, we want to find $\ell^{\mathbf s,\tilde{\mathbf m}}$ such that on $B(\mathbf s, 2r)$,
\begin{align}\label{eikon}
q_0\Big(x,i\big(\nabla W +\ell_0 \, \nabla \ell_0\big)\Big)\big(2\nabla W+\ell_0 \nabla \ell_0\big)\cdot \nabla\ell_0=0
\end{align}
and for $j \geq 1$
\begin{align}\label{transport}
 \bigg(
				q_0\Big(x,i\big(\nabla W +\ell_0 \, \nabla \ell_0\big)\Big)( 2\nabla W+\ell_0 \nabla \ell_0)+\ell_0\, U
				\bigg) \cdot \nabla \ell_j+ U \cdot \nabla \ell_0  \, \ell_j+R_j(\ell_0, \dots, \ell_{j-1})=0
\end{align}
where $U$ was introduced in \eqref{U}.
Note that Lemmas \ref{ureelle} and \ref{rreelle} ensure that the fact that the $(\ell_j)_{j\geq 0}$ are real valued is compatible with equations \eqref{transport}.

\subsection{Solving for $\ell_0^{\mathbf s,\tilde{\mathbf m}}$}
Denote $$p(x,\xi)= (-i \, \xi^t+\, \nabla W^t)q_0(x,\xi) (i\xi+\nabla W)= q_0(x,\xi) \, \xi \cdot \xi +  q_0(x,\xi)\, \nabla W \cdot \nabla W$$ the principal symbol of the whole operator $P_h$ and $\tilde p(x,\xi)=-p(x,i\xi)$ its complexification.
Notice that thanks to item \ref{q0id} from Hypothesis \ref{hyporw}, the quadratic approximation of $\tilde p$ at $(\mathbf s,0)$ coincides with the one of the complexification of the symbol of the Schr\"odinger operator $-h^2\Delta + |\nabla W|^2$ (up to a factor $\varrho$).
Hence, we get all the results from \cite{DimassiSjostrand}, chapter 3.
In particular, denoting
$$\Lambda_\pm=\Big\{(x,\xi)\, ; \lim_{t\to \mp \infty}\e^{tH_{\tilde p}}(x,\xi)=(\mathbf s,0)\Big\}$$
the stable manifolds associated to the Hamiltonian of ${\tilde p}$ near $(\mathbf s, 0)$, we obtain the following.
\begin{lem}\label{defpos}
There exist $\phi_\pm \in \mathcal C^\infty(B(\mathbf s,2r),\R)$ vanishing together with their gradients at $\mathbf s$ and such that
\begin{align}\label{lamphi}
\Lambda_\pm=\Big\{\big(x, \nabla\phi_\pm (x) \big) \, ; x\in B(\mathbf s,2r)\Big\}.
\end{align}
Moreover, the Hessian matrix of $\pm \phi_\pm$ at $\mathbf s$ is positive definite.
\end{lem}

\hip
At this point, one can proceed as in \cite{BonyLPMichel}, Lemmas 3.2 and 3.3 to establish the following Proposition after matching the notations by setting $\Lambda(\mathbf s)=2\varrho\, \mathcal W_s$, $b^0=0$, $A^0(\mathbf s)=\varrho \, \mathrm{Id}$ and
$B(\mathbf s)=0
$.

\begin{prop}\label{phinu}
Recall the notation \eqref{hessiennes}.
There exists $\ell_0^{\mathbf s,\tilde{\mathbf m}} \in \mathcal C^\infty(B(\mathbf s,2r),\R)$ such that 
\begin{enumerate}[label=\textbullet]
\item For $x \in B(\mathbf s,2r)$, 
$$\phi_+(x)=W(x)-W(\mathbf s)+\frac{\ell_0^{\mathbf s,\tilde{\mathbf m}}(x)^2}{2}.$$
In particular, $\ell_0^{\mathbf s,\tilde{\mathbf m}}$ vanishes at $\mathbf s$.
\item 
The function $\ell_0^{\mathbf s,\tilde{\mathbf m}}$ is a solution of \eqref{eikon} in $B(\mathbf s,2r)$.
\item The vector $\nabla \ell_0^{\mathbf s,\tilde{\mathbf m}} (\mathbf s)$ that we denote $\nu^{\mathbf s,\tilde{\mathbf m}}$ is not 0 and satisfies 
$$2 \mathcal W_s \nu^{\mathbf s,\tilde{\mathbf m}}=- \big|\nu^{\mathbf s,\tilde{\mathbf m}} \big|^2 \,\nu^{\mathbf s,\tilde{\mathbf m}}.$$
\item Finally,
$$\det \bigg(\mathrm{Hess}_{\mathbf s}\bigg(W+\frac{(\ell_0^{\mathbf s,\tilde{\mathbf m}})^2}{2}\bigg)\bigg)=\big|\det \mathcal W_{\mathbf s}\big|.$$
\end{enumerate}
\end{prop}


\subsection{Solving for $(\ell_j^{\mathbf s,\tilde{\mathbf m}})_{j\geq 1}$}

Once again we drop some exponents $\mathbf s$ and $\tilde{\mathbf m}$ for shortness.
Now that $\ell_0$ is given by Proposition \ref{phinu}, we can solve the transport equations \eqref{transport} by induction, so we suppose that $\ell_0,\dots,\ell_{j-1}$ are given and we want to find a solution $\ell_j$ to \eqref{transport}.
Denote 
$$\widetilde U=q_0\Big(x,i\big(\nabla W +\ell_0 \, \nabla \ell_0\big)\Big)( 2\nabla W+\ell_0 \nabla \ell_0)+\ell_0\, U\in \mathcal C^\infty(B(\mathbf s,2r))$$
and 
$$\tau=\nabla \ell_0 \cdot U \in \mathcal C^\infty(B(\mathbf s,2r))$$
where $U$ was introduced in \eqref{U}.
The function $\ell_j$ must satisfy $(\widetilde U \cdot \nabla +\tau) \ell_j=-R_j(\ell_0,\dots,\ell_{j-1})$ so we are intersted in the operaor $\mathcal L=\widetilde U \cdot \nabla +\tau$ that we decompose as $\mathcal L=\mathcal L_0^{\mathbf s}+\mathcal L_>$
with 
\begin{align*}
\mathcal L_0^{\mathbf s}
		=\widetilde U_0^{\mathbf s}\big(x-\mathbf s\big) \cdot \nabla+\tau_0^{\mathbf s}
\end{align*}
where $\widetilde U_0^{\mathbf s}$ is the differential of $\widetilde U$ at $\mathbf s$ and $\tau_0^{\mathbf s}=\tau(\mathbf s)$, that is with \eqref{hessiennes}
\begin{align}\label{alpha0} 
\widetilde U_0^{\mathbf s}=2\varrho \big( \mathcal W_s + \nu \nu ^t \big)  \qquad \text{and} \qquad  \tau_0^{\mathbf s}=\varrho \big|\nu^{\mathbf s,\tilde{\mathbf m}}\big|^2.
\end{align}
As usual, we will often omit the exponents $\mathbf s$ in the notations.
Notice that if we denote $\mathcal P^n_{hom}$ the space of homogeneous polynomials of degree $n$ in the variables $(x-\mathbf s)$, we have $\mathcal L_0 \in \mathscr L(\mathcal P^n_{hom})$
and for $P \in \mathcal P^n_{hom}$, $\mathcal L_>P(x) =O\big((x-\mathbf s)^{n+1}\big)$ near $\mathbf s$.
Using Proposition \ref{phinu}, it is easy to check that the spectrum of $\widetilde U_0^{\mathbf s}$ is exactly the spectrum of $2\varrho \mathcal W_s$ except that the negative eigenvalue $-\tau_0^{\mathbf s}$ is replaced by $\tau_0^{\mathbf s}$.
We can then apply Lemma A.1 from \cite{BonyLPMichel} to get that $\mathcal L_0^{\mathbf s}$ is invertible on $\mathcal P^n_{hom}$.
Thanks to this fact, one can proceed as in \cite{BonyLPMichel}, section 3.3 (see also \cite{DimassiSjostrand}, chapter 3), i.e find an approximate solution of \eqref{transport} using formal power series and then refine it into an actual solution using the characteristic method.
This gives the following result.
\begin{prop}
For all $j\geq 1$, there exists $\ell^{\mathbf s,\tilde{\mathbf m}}_j \in \mathcal C^\infty(B(\mathbf s,2r))$ solving \eqref{transport}.
Moreover, $\ell_j^{\mathbf s,\tilde{\mathbf m}}$ is real valued in view of Lemmas \ref{ureelle} and \ref{rreelle}.
\end{prop}
\vskip 0.07cm
\subsection{Construction of $\ell^{\mathbf s,\tilde{\mathbf m}}$}
Now that we have found $(\ell_j)_{j\geq 0} \subset \mathcal C^\infty(B(\mathbf s,2r),\R)$ solving \eqref{eikon} and \eqref{transport} with $\ell_0$ vanishing at $\mathbf s$, we can use a Borel procedure to construct $\ell \in \mathcal C^\infty(\R^d,\R)$ supported in $B(\mathbf s,3r)$ and satisfying $\ell \sim \sum_{j \geq 0}h^j \ell_j$ on $B(\mathbf s,2r)$. 
\begin{rema}\label{signl}
The properties \ref{hypol1}-\ref{hypoln-1} from \eqref{hypol} are satisfied by both the functions $\ell^{\mathbf s,\tilde{\mathbf m}}$ and $-\ell^{\mathbf s,\tilde{\mathbf m}}$.
Moreover, by Lemma \ref{rreelle}, $(-\ell_j^{\mathbf s,\tilde{\mathbf m}})_{j\geq 0}$ also solve \eqref{eikon} and \eqref{transport}. 
\end{rema}
\hip
At this point, a straightforward adaptation of the proof of Proposition 5.2 from \cite{me} yields the following result which states that all the properties from \eqref{hypol} are satisfied.

\begin{prop}\label{lexist}
We can choose the signs of the functions $(\ell^{\mathbf s,\tilde{\mathbf m}})_{\mathbf s \in \mathbf j^\alpha(\tilde{\mathbf m})}$ such that \eqref{hypol} holds true and the coefficients from the classical expansion of $\ell^{\mathbf s,\tilde{\mathbf m}}$ solve \eqref{eikon} and \eqref{transport}.
\end{prop}

\hip
We end this section with the following observation from \cite{BonyLPMichel} (Lemma 6.4).
\begin{lem}\label{l=-l}
If $\mathbf s\in \mathbf j^\alpha(\mathbf m)\cap \mathbf j^{\alpha'} (\mathbf m')$ with $\mathbf m\neq \mathbf m'$, we can suppose (up to a modification by $O(h^\infty)$) that
$$\ell^{\mathbf s,\mathbf m}=-\ell^{\mathbf s,\mathbf m'}$$
and consequently,
$$\theta^\alpha_{\mathbf m,h}=1-\theta^{\alpha'}_{\mathbf m',h}\qquad \text{on } B(\mathbf s,r)\cap \left(\mathrm{supp}\, \theta^\alpha_{\mathbf m,h}\cup \mathrm{supp}\, \theta^{\alpha'}_{\mathbf m',h} \right).$$
\end{lem}
\hip

\section{Interaction between two wells}

Let $\alpha$, $\alpha'\in \mathcal A$ as well as $\mathbf m\in \mathcal U^{(0)}_\alpha$ and $\mathbf m'\in \mathcal U^{(0)}_{\alpha'}$.

\begin{lem}\label{alpha=alpha'}
For all $\tilde {\mathbf m}\in \widehat{\mathcal U}^{(0)}_\alpha$ and $\tilde{\mathbf m}'\in \widehat{\mathcal U}^{(0)}_{\alpha'}$ such that $\mathbf j^\alpha(\tilde{\mathbf m})\cap \mathbf j^{\alpha'}(\tilde{\mathbf m}')\neq \emptyset$, the following holds:
$$\alpha= \alpha' \qquad \text{or} \qquad \varphi^\alpha_{\mathbf m}(\tilde{\mathbf m}) \varphi^{\alpha'}_{\mathbf m'}(\tilde{\mathbf m}')=0.$$
\end{lem}

\begin{proof}
First, notice that since $\widehat{\mathcal U}^{(0)}_{\alpha}\subset E_-(\mathbf m)$ and $\widehat{\mathcal U}^{(0)}_{\alpha'}\subset E_-(\mathbf m')$, our hypothesis implies that $ E_-(\mathbf m)=  E_-(\mathbf m')$.
In particular, $\widehat{\mathbf m}'=\widehat{\mathbf m}$.
If $\widehat{\mathbf m}\notin \{ \tilde{\mathbf m}, \, \tilde{\mathbf m}' \}$, we easily have $\tilde{\mathbf m} \mathcal R \tilde{\mathbf m}'$ and $\alpha=\alpha'$.
Let us now suppose that 
\begin{align}\label{mchapeaumtilde}
\widehat{\mathbf m}\in \{ \tilde{\mathbf m}, \, \tilde{\mathbf m}' \}.
\end{align}
According to Lemma 3.4 from \cite{me}, $\mathbf m$ and $\mathbf m'$ are in the same CC of $\{W\leq \boldsymbol \sigma (\mathbf m)\}$ that we denote $E_\leqslant$.
By uniqueness of $\widehat{\mathbf m}$ in $E_-(\mathbf m)$, each CC of $\{W<\boldsymbol \sigma(\mathbf m)\}\cap E_\leqslant$ contains exactly one element from $\boldsymbol \sigma^{-1} (\{\boldsymbol \sigma(\mathbf m)\}) \cup \{\widehat{\mathbf m}\}$.
If $\mathbf m$ or $\mathbf m'$ is of type II, we get by definition of $\mathcal R$ that $\alpha=\alpha'$.
Otherwise, \eqref{mchapeaumtilde} combined with item \ref{phidelta} from Lemma \ref{phiexist} yield $\varphi^\alpha_{\mathbf m}(\tilde{\mathbf m}) \varphi^{\alpha'}_{\mathbf m'}(\tilde{\mathbf m}')=0.$
\end{proof}

\hip
With the notations from Section \ref{sectquasim}, let us denote for $\tilde {\mathbf m}\in \widehat{\mathcal U}^{(0)}_\alpha$ and $\tilde{\mathbf m}'\in \widehat{\mathcal U}^{(0)}_{\alpha'}$
\begin{align}
\mathcal N^{\alpha,\alpha'}_{\tilde {\mathbf m}, \tilde {\mathbf m}'}=h^{-d/2}c^\alpha_h(\tilde{\mathbf m}) c^{\alpha'}_h(\tilde{\mathbf m}')\e^{\frac{W(\tilde{\mathbf m})+W(\tilde{\mathbf m}')}{h}}  \left\langle P_h \left((a^h)^{-1}\chi_\alpha\theta^{\alpha}_{\tilde{\mathbf m},h}\e^{-W/h}\right), (a^h)^{-1}\chi_{\alpha'} \theta^{\alpha'}_{\tilde{\mathbf m}',h}\e^{-W/h} \right\rangle.
\end{align}
When $\alpha=\alpha'$, we denote for shortness $\mathcal N^{\alpha,\alpha}_{\tilde {\mathbf m}, \tilde {\mathbf m}'}=\mathcal N^{\alpha}_{\tilde {\mathbf m}, \tilde {\mathbf m}'}$.

\begin{lem}\label{N0}
Let $\tilde {\mathbf m}\in \widehat{\mathcal U}^{(0)}_\alpha$ and $\tilde{\mathbf m}'\in \widehat{\mathcal U}^{(0)}_{\alpha'}$.
\begin{enumerate}[label=\textbullet]
\item If $\mathbf j^\alpha(\tilde{\mathbf m})\cap \mathbf j^{\alpha'}(\tilde{\mathbf m}')= \emptyset$, we have
$$\mathcal N^{\alpha,\alpha'}_{\tilde {\mathbf m}, \tilde {\mathbf m}'}=O\Big(h^\infty \e^{-\frac{\boldsymbol \sigma (\alpha)+\boldsymbol \sigma (\alpha')-W(\tilde{\mathbf m})-W(\tilde{\mathbf m}')}{h}}\Big).$$
\item When $\alpha=\alpha'$, we have
$$\mathcal N^{\alpha}_{\tilde {\mathbf m}, \tilde {\mathbf m}'}=h \e^{-\frac{2\boldsymbol \sigma (\alpha)-W(\tilde{\mathbf m})-W(\tilde{\mathbf m}')}{h}} N^{\alpha}_{\tilde {\mathbf m}, \tilde {\mathbf m}'}$$
with $N^{\alpha}_{\tilde {\mathbf m}, \tilde {\mathbf m}'}$ admitting an asymptotic expansion whose first term is
$$N_{\tilde {\mathbf m}, \tilde {\mathbf m}'}^{\alpha,0}=  \frac{(-1)^{1-\delta_{\tilde{\mathbf m}, \tilde{\mathbf m}'}}}{2\pi}  \Big(\sum_{\mathbf m \in H^\alpha(\tilde{\mathbf m})} \det\mathcal W_{\mathbf m}^{-1/2}\Big)^{-1/2} \Big(\sum_{\mathbf m' \in H^{\alpha}(\tilde{\mathbf m}')} \det\mathcal W_{\mathbf m'}^{-1/2}\Big)^{-1/2} \sum_{\mathbf s\in \mathbf j^\alpha(\tilde{\mathbf m})\cap \mathbf j^{\alpha}(\tilde{\mathbf m}')} |\det \mathcal W_{\mathbf s}|^{-1/2}   \tau_0^{\mathbf s}$$
where we recall the notation \eqref{hessiennes} and that $-\tau_0^{\mathbf s}$ is the negative eigenvalue of $2\varrho \,\mathcal W_{\mathbf s}$.
\end{enumerate}
\end{lem}

\begin{proof}
First, notice that by Hypothesis \ref{hyporw}, we have
$$\left\langle P_h \left((a^h)^{-1}\chi_\alpha\theta^{\alpha}_{\tilde{\mathbf m},h}\e^{-W/h}\right), (a^h)^{-1}\chi_{\alpha'} \theta^{\alpha'}_{\tilde{\mathbf m}',h}\e^{-W/h} \right\rangle=\left\langle \tilde P_h \left(\chi_\alpha\theta^{\alpha}_{\tilde{\mathbf m},h}\e^{-W/h}\right), \chi_{\alpha'} \theta^{\alpha'}_{\tilde{\mathbf m}',h}\e^{-W/h} \right\rangle.$$
We will use the following localizations and estimates obtained thanks to \eqref{suppchi} and \eqref{nablatheta}:
\begin{align}\label{dwfloc}
d_W\big(\chi_\alpha\theta^\alpha_{\tilde{\mathbf m},h}\e^{-W/h}\big)\,\1_{\big\{\mathrm{dist}\big(\cdot,\, \mathbf j^\alpha(\widehat{\mathcal U}^{(0)}_\alpha)\big)\geq r\big\}}=O\Big(h^\infty \e^{-\frac{\boldsymbol \sigma(\alpha)}{h}}\Big) ;
\end{align}
\begin{align}\label{dwf=o}
d_W\big(\chi_\alpha\theta^\alpha_{\tilde{\mathbf m},h}\e^{-W/h}\big) =O\Big(h^{\frac{2+d}{4}}\e^{-\frac{\boldsymbol \sigma(\alpha)}{h}}\Big) 
\end{align}
and
by the non stationnary phase applied as for \eqref{vgrand}
\begin{align}
\mathrm{Op}_h(q^h)\Big(d_W \big(\chi_\alpha\theta^\alpha_{\tilde{\mathbf m},h}\e^{-W/h}\big) \Big) \,\1_{\big\{\mathrm{dist}\big(\cdot,\, \mathbf j^\alpha(\widehat{\mathcal U}^{(0)}_\alpha)\big)\geq 2r\big\}}=O\Big(h^\infty \e^{-\frac{\boldsymbol \sigma(\alpha)}{h}}\Big).\label{opadwfloc}
\end{align}
Using the factorized structure of $\tilde P_h$, the boundedness of $\mathrm{Op}_h(q^h)$ as well as \eqref{dwfloc}, \eqref{dwf=o} and \eqref{opadwfloc}, we can write that 
\begin{align}\label{interact1}
 \Big\langle \tilde P_h \big(\chi_\alpha\theta^\alpha_{\tilde{\mathbf m},h}\e^{-W/h}\big), \chi_{\alpha'}\theta^{\alpha'}_{\tilde{\mathbf m}',h}&\e^{-W/h} \Big\rangle+O\left(h^\infty \e^{-\frac{\boldsymbol{\sigma}(\alpha)+\boldsymbol{\sigma}(\alpha')}{h}}\right)\nonumber\\
		&=\sum_{\mathbf s\in \mathbf j^\alpha(\tilde{\mathbf m})\cap \mathbf j^{\alpha'}(\tilde{\mathbf m}')} \Big\langle \mathrm{Op}_h(q^h)\Big(d_W \big(\chi_\alpha\theta^\alpha_{\tilde{\mathbf m},h}\e^{-W/h}\big) \Big) , d_W\big(\chi_{\alpha'}\theta^{\alpha'}_{\tilde{\mathbf m}',h}\e^{-W/h}\big) \Big\rangle_{\mathbf s}
\end{align}
where $\langle \cdot , \cdot \rangle_{\mathbf s}$ denotes the inner product on $L^2(B(\mathbf s,r))$.
This already proves the first statement.
Now when $\alpha=\alpha'$, thanks to the fact that $\e^{-W/h}\in \mathrm{Ker}$ $d_W$ and by Lemma \ref{l=-l}, we have for $\mathbf s\in \mathbf j^\alpha(\tilde{\mathbf m})\cap \mathbf j^{\alpha}(\tilde{\mathbf m}')$
\begin{align}\label{Ph1-}
d_W \big(\chi_\alpha\theta^\alpha_{\tilde{\mathbf m},h}\e^{-W/h}\big)&=d_W \big((\chi_\alpha\theta^\alpha_{\tilde{\mathbf m},h}-1)\e^{-W/h}\big)=d_W \big(\chi_\alpha(\theta^\alpha_{\tilde{\mathbf m},h}-1)\e^{-W/h}\big)+O\big(h^\infty \e^{-\frac{\boldsymbol{\sigma}(\alpha)}{h}}\big)\\
		&=(-1)^{1-\delta_{\tilde{\mathbf m}, \tilde{\mathbf m}'}}d_W \big(\chi_\alpha\theta^{\alpha}_{\tilde{\mathbf m}',h}\e^{-W/h}\big) +O\big(h^\infty \e^{-\frac{\boldsymbol{\sigma}(\alpha)}{h}}\big) \qquad \text{on } B(\mathbf s,r).
\end{align}
Thus, \eqref{interact1} becomes 
\begin{align}\label{interact1.5}
\Big\langle \tilde P_h \big(\chi_\alpha\theta^\alpha_{\tilde{\mathbf m},h}&\e^{-W/h}\big), \chi_\alpha\theta^{\alpha}_{\tilde{\mathbf m}',h}\e^{-W/h} \Big\rangle+O\big(h^\infty \e^{-\frac{2\boldsymbol{\sigma}(\alpha)}{h}}\big)\nonumber\\
			&=(-1)^{1-\delta_{\tilde{\mathbf m}, \tilde{\mathbf m}'}}\sum_{\mathbf s\in \mathbf j^\alpha(\tilde{\mathbf m})\cap \mathbf j^{\alpha}(\tilde{\mathbf m}')}  \Big\langle \mathrm{Op}_h(q^h)\Big(d_W \big(\chi_\alpha\theta^\alpha_{\tilde{\mathbf m},h}\e^{-W/h}\big) \Big) , d_W\big(\chi_{\alpha}\theta^{\alpha}_{\tilde{\mathbf m},h}\e^{-W/h}\big) \Big\rangle_{\mathbf s}.
\end{align}
We can now work as in \cite{me}, proof of Lemma 5.3, to get that
$$ \Big\langle \mathrm{Op}_h(q^h)\Big(d_W \big(\chi_\alpha\theta^\alpha_{\tilde{\mathbf m},h}\e^{-W/h}\big) \Big) , d_W\big(\chi_{\alpha}\theta^{\alpha}_{\tilde{\mathbf m},h}\e^{-W/h}\big) \Big\rangle_{\mathbf s} = \frac{h}{2\pi} (\pi h)^{d/2} \e^{-\frac{2\boldsymbol \sigma(\alpha)}{h}}\vartheta^{\mathbf s,h} $$
with $\vartheta^{\mathbf s,h}$ admitting a classical expansion whose first term is $|\det \mathcal W_{\mathbf s}|^{-1/2} \, \tau_0^{\mathbf s}$.
Combining this with \eqref{expch}, we get the announced result.
\end{proof}

\section{Interaction between two quasimodes}
\hip
By \eqref{quasim}, we have for $\mathbf m\in \mathcal U^{(0)}_\alpha$ and $\mathbf m' \in \mathcal U^{(0)}_{\alpha'}$
\begin{align}\label{pff1}
\langle P_h f_{\mathbf m,h},  f_{\mathbf m',h} \rangle&= h^{-d/2}\sum_{\tilde{\mathbf m} \in \widehat{\mathcal U}^{(0)}_\alpha} \sum_{\tilde{\mathbf m}' \in \widehat{\mathcal U}^{(0)}_{\alpha'}} \varphi^\alpha_{\mathbf m}(\tilde{\mathbf m}) \varphi^{\alpha'}_{\mathbf m'}(\tilde{\mathbf m}')c_h^\alpha(\tilde{\mathbf m})c_h^{\alpha'}(\tilde{\mathbf m}')\\
		&\qquad \qquad \qquad \qquad \qquad  \times \left\langle P_h \left((a^h)^{-1}\chi_\alpha\theta^{\alpha}_{\tilde{\mathbf m},h}\e^{-W_{\mathbf m}/h}\right), (a^h)^{-1}\chi_{\alpha'} \theta^{\alpha'}_{\tilde{\mathbf m}',h}\e^{-W_{\mathbf m'}/h} \right\rangle\\
		&= \sum_{\tilde{\mathbf m} \in \widehat{\mathcal U}^{(0)}_\alpha} \sum_{\tilde{\mathbf m}' \in \widehat{\mathcal U}^{(0)}_{\alpha'}} \varphi^\alpha_{\mathbf m}(\tilde{\mathbf m}) \varphi^{\alpha'}_{\mathbf m'}(\tilde{\mathbf m}')\mathcal N^{\alpha,\alpha'}_{\tilde{\mathbf m}, \tilde{\mathbf m}'} \\
		&=\mathop{\sum_{\tilde{\mathbf m} \in \widehat{\mathcal U}^{(0)}_\alpha;\vphantom{\widehat{\mathcal U}^{(0)}_{\alpha'}}}}_{\substack{W(\tilde{\mathbf m})=W(\mathbf m)}}\quad \mathop{\sum_{\tilde{\mathbf m}' \in \widehat{\mathcal U}^{(0)}_{\alpha'};}}_{\substack{W(\tilde{\mathbf m}')=W(\mathbf m')}}  \varphi^\alpha_{\mathbf m}(\tilde{\mathbf m}) \varphi^{\alpha'}_{\mathbf m'}(\tilde{\mathbf m}')\mathcal N^{\alpha,\alpha'}_{\tilde{\mathbf m}, \tilde{\mathbf m}'} 
\end{align}
by item \ref{phinonnul} from Lemma \ref{phiexist}.
According to Lemma \ref{N0}, the leading terms in the previous sum are the ones for which $\tilde{\mathbf m}$, $\tilde{\mathbf m}'$ are such that $\mathbf j^\alpha(\tilde{\mathbf m})\cap \mathbf j^{\alpha'}(\tilde{\mathbf m}')\neq \emptyset$.
Combined with Lemma \ref{alpha=alpha'}, we get
\begin{align}\label{pff2}
\langle P_h f_{\mathbf m,h},  f_{\mathbf m',h} \rangle=\delta_{\alpha,\alpha'}  \sum_{\tilde{\mathbf m},\tilde{\mathbf m}' \in \widehat{\mathcal U}^{(0)}_{\alpha}} \varphi^\alpha_{\mathbf m}(\tilde{\mathbf m}) \varphi^{\alpha}_{\mathbf m'}(\tilde{\mathbf m}')\mathcal N^{\alpha}_{\tilde{\mathbf m}, \tilde{\mathbf m}'} +O\big(h^\infty \e^{-\frac{S(\mathbf m)+S(\mathbf m')}{h}}\big).
\end{align}

We now want to go from quasimodes to actual eigenfunctions.
This is where the optimization on the choice of the functions $\ell^{\mathbf s,\mathbf m}$ will enable us to have the correct error terms.
Here we briefly remind the procedure and give the main arguments.
We refer to \cite{LPMichel} for more details.
First, combining Proposition \ref{phf}, item \ref{hypoln-1} from \eqref{hypol}, Proposition \ref{lexist} and a standard Laplace method, we obtain the following fundamental estimate.

\begin{lem}\label{Pf^2}
Let $\mathbf m \in \underline{\mathcal U}^{(0)}$.
We have 
$$\|P_hf_{\mathbf m,h}\|=O\big(h^\infty \e^{-\frac{S(\mathbf m)}{h}}\big).$$
\end{lem}
\hip
Now, considering the orthogonal projector on the generalized eigenspace associated to the small eigenvalues of $P_h$ given by
\begin{align}\label{Pi0}
\Pi_0=\frac{1}{2i\pi}\int_{|z|=c h}(z-P_h)^{-1}\D z
\end{align}
and writing
$$(1-\Pi_0) f_{\mathbf m,h}=\frac{-1}{2i\pi}\int_{|z|= c h}z^{-1}(z-P_h)^{-1}P_h f_{\mathbf m,h}\D z,$$
Lemma \ref{Pf^2} together with the resolvent estimate \eqref{resest} give that
for any $\mathbf m\in \mathcal U^{(0)}$, we have 
$$\|(1-\Pi_0)f_{\mathbf m,h}\|=O\big(h^\infty \e^{-\frac{S(\mathbf m)}{h}}\big).$$
Proceeding as in Proposition 4.10 from \cite{LPMichel}, one can then establish the following thanks to Proposition \ref{ortho}.
\begin{lem}\label{Pi0fon}
The family $(\Pi_0f_{\mathbf m,h})_{\mathbf m\in \mathcal U^{(0)}}$ is almost orthonormal:
there exists $c>0$ such that 
$$\langle \Pi_0f_{\mathbf m,h}, \Pi_0f_{\mathbf m',h}  \rangle=\delta_{\mathbf m, \mathbf m'}+O(\e^{-c/h}).$$
In particular, it is a basis of the space $H=\mathrm{Ran}\, \Pi_0$ introduced in \eqref{Pi0}.\\
Moreover, we have
$$\langle P_h\Pi_0f_{\mathbf m,h}, \Pi_0f_{\mathbf m',h}  \rangle= \langle P_hf_{\mathbf m,h}, f_{\mathbf m',h}  \rangle +O\big(h^\infty \e^{-\frac{S(\mathbf m)+S(\mathbf m')}{h}}\big).$$
\end{lem}

Let us re-label the local minima $\mathbf m_1,\dots, \mathbf m_{n_0}$ so that $(S(\mathbf m_j))_{j=1,\dots,n_0}$ is non increasing in $j$.
For shortness, we will now denote 
$$f_j=f_{\mathbf m_j,h}$$
which still depends on $h$.
We also denote $(\tilde u_j)_{j=1,\dots,n_0}$ the orthogonalization by the Gram-Schmidt procedure of the family $(\Pi_0f_j)_{j=1,\dots,n_0}$ and 
$$u_j=\frac{\tilde u_j}{\|\tilde u_j\|}.$$
In this setting and with our previous results, we get the following (see \cite{LPMichel}, Proposition 4.12 for a proof).

\begin{lem}\label{Puu'}
For all $1\leq j,k \leq n_0$, it holds 
$$\langle P_h u_j, u_k \rangle=\langle P_hf_j, f_k  \rangle +O\big(h^\infty \e^{-\frac{S(\mathbf m_j)+S(\mathbf m_k)}{h}}\big).$$
\end{lem}
\hip
In order to compute the small eigenvalues of $P_h$, let us now consider the restriction $P_h|_H:H\to H$.
We denote $\hat u_j=u_{n_0-j+1}$ and $\mathcal M$ the matrix of $P_h|_H$ in the orthonormal basis $(\hat u_1,\dots,\hat u_{n_0})$.
Since $\hat u_{n_0}=u_1=f_1$, we have 
$$\mathcal M=\begin{pmatrix}
		\mathcal M'&0\\
		0&0
\end{pmatrix} \qquad \text{where }\qquad \mathcal M'=\Big(\langle P_h\hat u_{j}, \hat u_k \rangle\Big)_{1\leq j,k\leq n_0-1}$$
and it is sufficient to study the spectrum of $\mathcal M'$.
We will also denote $\{\hat S_1< \dots < \hat S_p\}$ the set $\{S(\mathbf m_j)\, ; \, 2\leq j \leq n_0\}$ and for $1\leq k \leq p$, $E_k$ the subspace of $L^2(\R^{d})$ generated by $\{\hat u_r\, ; \,  S(\mathbf m_r)=\hat S_k\}$.
Finally, we set $\varpi_k=\e^{-(\hat S_k-\hat S_{k-1})/h}$ for $2\leq k \leq p$ and $\varepsilon_j(\varpi)=\prod_{k=2}^j \varpi_k=\e^{-(\hat S_j-\hat S_1)/h}$ for $2\leq j \leq p$ (with the convention $\varepsilon_1(\varpi)=1$).

For a given class $\alpha\in \mathcal A$, let us denote $n_\alpha=|\mathcal U^{(0)}_\alpha|$ and also label its elements $\mathbf m^\alpha_1,\dots, \mathbf m^\alpha_{n_\alpha}$ so that $(S(\mathbf m^\alpha_j))_{j=1,\dots,n_\alpha}$ is non decreasing in $j$.
We also set $\mathbf m^\alpha_{n_\alpha+1}=\widehat{\mathbf m}$ for some $\mathbf m\in \mathcal U^{(0)}_\alpha$.
We 
will consider the matrix 
\begin{align}\label{Malpha}
M^\alpha_h=\big(N^\alpha \varphi^\alpha_{\mathbf m^\alpha_{j}} \cdot  \varphi^\alpha_{\mathbf m^\alpha_{k}}\big)_{1\leq j,k\leq n_\alpha}=\mathcal T_\alpha^* N^\alpha \mathcal T_\alpha
\end{align}
where $N^\alpha$ is the matrix introduced in Lemma \ref{N0} and
\begin{align}\label{T}
\mathcal T_\alpha = \big(\varphi^\alpha_{\mathbf m^\alpha_k}(\mathbf m^\alpha_j)\big)\mathop{}_{\substack{1\leq j \leq n_\alpha+1 \\  1\leq k \leq n_\alpha}}.
\end{align}
Before we can state our main result, we need to introduce some material from \cite{BonyLPMichel}.
For the finite dimensional vector space $E = E_{1} \oplus \cdots \oplus E_{p}$, and $j \in \{ 1 , \dots , p \}$, let us write a general matrix $M \in \mathcal M ( E )$ by blocks
\begin{equation} \label{b9}
M = \left( \begin{array}{cc}
A & B \\
C & D
\end{array} \right) : ( E_{1} \oplus \cdots \oplus E_{j - 1} ) \oplus ( E_{j} \oplus \cdots \oplus E_{p} ) \longrightarrow ( E_{1} \oplus \cdots \oplus E_{j - 1} ) \oplus ( E_{j} \oplus \cdots \oplus E_{p} ) .
\end{equation}
If $A \in \mathcal M ( E_{1} \oplus \cdots \oplus E_{j-1} )$ is invertible, the Schur matrix of $M$ (with respect to the vector space $E_{1} \oplus \cdots \oplus E_{j - 1}$) is the matrix on $E_{j} \oplus \cdots \oplus E_{p}$ defined by
\begin{equation*}
\mathcal R_{j} ( M ) = D - C A^{- 1} B ,
\end{equation*}
where by convention $\mathcal R_{1} ( M ) = M$. By the Schur complement method, $M$ is invertible if and only if $\mathcal R_{j} ( M )$ is invertible. 
We will also denote by $\mathcal J : \mathcal M ( \oplus_{k = j , \ldots , p } E_{k} ) \to \mathcal M ( E_{j})$ the restriction map to the first vector space $E_{j}$ of $\oplus_{k = j , \ldots , p } E_{k}$. More precisely, with the notations of \eqref{b9}, we will write $\mathcal J ( M ) = A$ when $j=1$. Of course, the map $\mathcal J$ depends on $j \in \{ 1 , \dots , p \}$, but we omit this dependence since the set on which $\mathcal J$ is acting will be obvious in the sequel.
We will also use the convention 
$$\mathrm{Spec} \big( \mathcal J \circ \mathcal R_j(M^\alpha_h) \big)=\emptyset \quad \text{if} \quad\hat S_j \notin \{S(\mathbf m^\alpha_k)\, ; \, k=1,\dots,n_\alpha \}.$$

\begin{thm}\label{thmgene}
With the notations introduced above, we have
$$\mathrm{Spec}(\mathcal M')\subset h \e^{-2\hat S_1/h} \bigcup_{\alpha\in \mathcal A} \bigcup_{j=1}^p \varepsilon_j(\varpi)^2 \Big( \mathrm{Spec} \big( \mathcal J \circ \mathcal R_j(M^\alpha_h) \big) + D\big(0,O(h^\infty)\big)\Big) $$
with $M^\alpha_h$ admitting an asymptotic expansion whose first term is $\mathcal T^*_0 N^{\alpha,0} \mathcal T_0$
where $N^{\alpha,0}$ is defined in Lemma \ref{N0}
and $\mathcal T_0$ is the leading term of the matrix $\mathcal T_\alpha$ given by Lemma \ref{phiexist}.
\end{thm}

\begin{proof}
Consider the symmetric matrix $M^\#_h\in \mathcal M_{n_0-1}(\R)$ defined by
\begin{align}\label{Msharp}
(M^\#_h)_{j,k}=
\begin{cases}
    N^\alpha \varphi^\alpha_{\mathbf m_{n_0-j+1}} \cdot  \varphi^\alpha_{\mathbf m_{n_0-k+1}}     & \quad \text{if } \mathbf m_{n_0-j+1} \, ,\; \mathbf m_{n_0-k+1} \in \mathcal U^{(0)}_\alpha\\
   0  & \quad \text{otherwise} 
  \end{cases}
\end{align}
and notice that in view of Lemma \ref{Puu'} and \eqref{pff2}, we have
$$h^{-1}\e^{2\hat S_1/h}\mathcal M'=\Omega(\varpi) \big(M^\#_h+O(h^\infty)\big) \Omega(\varpi)$$
where $\Omega(\varpi)=\mathrm{diag}\big(\varepsilon_1(\varpi)\mathrm{Id}_{E_1},\dots,\varepsilon_p(\varpi)\mathrm{Id}_{E_p}\big)$.
Clearly, $M^\alpha_h$ is the restriction to $\alpha$ of the matrix $M^\#_h$ which is permutation similar to the block-diagonal matrix $\mathrm{diag}\big(M^\alpha_h \, ; \, \alpha\in \mathcal A \big)$.
In particular, $M^\#_h$ admits an asymptotic expansion thanks to Lemmas \ref{phiexist} and \ref{N0}.
Moreover, it is positive definite 
as each $M^\alpha_h$ appears to be positive definite.
Indeed, 
$\mathcal T_0$ is clearly injective as the family $(\varphi^\alpha_{\mathbf m})_{\mathbf m\in \mathcal U^{(0)}_\alpha}$ is orthonormal and it is shown in \cite{BonyLPMichel}, Proposition 6.8, that $N^{\alpha,0}=L_\alpha^*L_\alpha$ where $L_\alpha$ is an injective matrix, so $M_0^\alpha$ is positive definite.
In the words of Definition 6.7 from \cite{BonyLPMichel}, we obtain that $h^{-1}\e^{2\hat S_1/h}\mathcal M'$ is a classical graded symmetric matrix so we can apply Theorem 4 from \cite{BonyLPMichel} to get 
$$\mathrm{Spec}(\mathcal M')\subset h \e^{-2\hat S_1/h}  \bigcup_{j=1}^p \varepsilon_j(\varpi)^2 \Big( \mathrm{Spec} \big( \mathcal J \circ \mathcal R_j(M^\#_h) \big) + D\big(0,O(h^\infty)\big)\Big).$$
We can then conclude as
$$\mathrm{Spec} \big( \mathcal J \circ \mathcal R_j(M^\#_h) \big)=\bigcup_{\alpha\in \mathcal A} \mathrm{Spec} \big( \mathcal J \circ \mathcal R_j(M^\alpha_h) \big)$$
(see \cite{BonyLPMichel}, Theorem 6 and above for details).
\end{proof}

\addtocontents{toc}{\SkipTocEntry}
\subsection*{Acknowledgements}
The author is grateful to Laurent Michel for his advice through this work as well as Jean-François Bony for helpful discussions.\\
This work is supported by the ANR project QuAMProcs 19-CE40-0010-01.

\nocite{}
\bibliography{bibBoltz} 
\bigskip
\scshape \small Thomas Normand, Laboratoire de Math\'ematiques Jean Leray, Universit\'e de Nantes
\normalfont

\end{document}